\documentclass[a4paper,11pt]{amsart}
\usepackage{url}
\usepackage[colorlinks=false]{hyperref}
\usepackage{amssymb,amscd,amsmath,amsthm,amsfonts,color}
\usepackage[all]{xy}
\usepackage[pdftex]{graphicx}

\newcommand\boxb[1]{\square_b}

\numberwithin{equation}{section}

\newcommand\paperbody%
        {}

\newcommand{\cK}{{\mathcal K}}

\newtheorem{proposition}{Proposition}

\newtheorem{theorem}{Theorem}
\newtheorem{non-theorem}{Non-Theorem}

\theoremstyle{remark}
\newtheorem{definition}{Definition}
\newtheorem{remark}{Remark}

\newcommand{\cT}{\mathcal{T}}
\newcommand{\cL}{\mathcal{L}}

\newcommand{\cS}{\mathcal{S}}
\newcommand{\calU}{\mathcal{U}}
\newcommand{\calN}{\mathcal{N}}
\newcommand{\cX}{\mathcal{X}}

\newcommand{\ba}{\begin{array}}
\newcommand{\ea}{\end{array}}

\begin{document}
\title[Numerical metrics on moduli spaces of Calabi-Yau manifolds]
{Numerical Weil-Petersson metrics on moduli spaces of Calabi-Yau manifolds}

\author{Julien Keller and Sergio Lukic}
\address{ \newline Corresponding address for Julien Keller:  CMI, 39 rue Fr\'ed\'eric Joliot-Curie 13453 Marseille, France.
\newline Corresponding addresses for Sergio Luki\' c: Department of Mathematics, Imperial College, 180 Queen's Gate, London, U.K. 
\newline
Current address:  School of Natural Sciences, Institute for Advanced Study, Princeton NJ 08540, USA.
}
\email{jkeller@cmi.univ-mrs.fr, sergio.lukic@gmail.com}

\begin{abstract} 
We introduce a simple and very fast algorithm to compute Weil-Petersson metrics on moduli spaces of Calabi-Yau varieties. 
Additionally, we introduce a second algorithm to approximate the same metric using Donaldson's quantization link between
infinite and finite dimensional Geometric Invariant Theoretical (GIT) quotients that describe moduli spaces of varieties. Although this second
algorithm is slower and more sophisticated, it can also be used to compute similar 
metrics on other moduli spaces (e.g. moduli spaces of vector bundles on Calabi-Yau varieties). 
We study the convergence properties of both algorithms and provide explicit computer implementations using a family of Calabi-Yau quintic hypersurfaces 
in $\mathbb{P}^4$.
Also, we include discussions on: the existing methods that are used to compute this class of metrics,
the background material that we use to build our algorithms, and how to extend the second algorithm to the vector bundle case.
\end{abstract}


\maketitle

\setcounter{tocdepth}{3}
%\tableofcontents

\paperbody

\section{Introduction}

Exploiting the geometric and analytic properties of complex differential manifolds has allowed researchers to attain several impressive results regarding
the existence of solutions of some difficult non-linear PDEs. In particular, Yau's proof of Calabi's 
conjecture \cite{Yau} and Donaldson-Uhlenbeck-Yau's proof of the existence of Hermite-Einstein metrics on stable holomorphic vector bundles, have become
classical examples of how one can connect several complex-geometric objects with solutions of non-linear PDEs.
Although one rarely expects to find exact analytic solutions of such PDEs, one can study their analytic properties by means of
numerical methods. For instance, several numerical techniques introduced by Donaldson \cite{Don} allowed, for the first time, the computation of numerical approximations 
to K\"ahler-Einstein and Hermite-Einstein metrics on complex algebraic manifolds.  In this paper, we extend those techniques to develop numerical methods that we use to compute Weil-Petersson (WP) metrics on moduli spaces of such classes of metrics. Our main focus here is the Weil-Petersson metric on moduli spaces 
of complex structures on polarized Calabi-Yau manifolds. However, we will also discuss how some of our results can be extended to more general moduli spaces. 

The motivation for finding numerical approximations of WP metrics can be found in different sources. For instance, a source of motivation comes from the study of the global Weil-Petersson geometry on moduli spaces of Calabi-Yau manifolds, \cite{DL}. In this case, one of the algorithms introduced in this paper should allow us to estimate Weil-Petersson volumes of moduli spaces in a sensible amount of time and with reasonable precision.
Another source of motivation is the program by Douglas et. al. \cite{Dou}
to numerically compute K\"ahler metrics that appear in effective field theories associated with Calabi-Yau compactifications of string theory.
In the second part of this introduction, we outline the basic calculation of
the effective field theory action associated with these Calabi-Yau compatifications. In particular, we show how a
full characterization of this field theory action involves the computation of Weil-Petersson metrics on spaces of Ricci-flat metrics and/or HYM connections.

The remainder of the paper is organized as follows.
In section 2 we review some general results on moduli spaces of polarized
Calabi-Yau manifolds, define their corresponding Weil-Petersson metrics, and compute a particular example using an exact closed-form expression.
By combining formulas of local deformations of the holomorphic top form under diffeomorphisms, Monte Carlo integration techniques, and 
building on previous theoretical work (e.g. \cite{Tian}, \cite{Todorov}), we introduce a fast algorithm that computes the WP metric in section 3.
We apply this algorithm to evaluate the WP metric on a one-parameter family of Calabi-Yau quintic three-folds
in $\mathbb{P}^4$.
In section 4, we review basic concepts on moduli spaces of polarized varieties from the point of view of Geometric Invariant Theory. 
We introduce a natural sequence of K\"ahler metrics on the moduli space of Calabi-Yau manifolds, and use it  
to approximate the WP metric on the moduli space of Calabi-Yau manifolds. We implement the associated algorithm in the particular case
of the family of quintic three-folds and show how the sequence of metrics converges to the 
Weil-Petersson metric on the moduli space of Ricci-flat metrics.

\subsection{Calabi-Yau compactifications in string theory}

Several approaches that attempt to unify particle physics at high energies employ elegant field theories defined on high dimensional manifolds, such as $X\times \mathbb{R}^4$, 
where $\mathbb{R}^4$ denotes the standard four-dimensional space-time and $X$ denotes a spacelike manifold of dimension larger than zero. 
In this formalism, one recovers four-dimensional physics in the small-size limit of $X$. 
This idea, known as compactification of a field theory, has inspired much work in the interface between geometry and physics. 
In this subsection, we review how evaluating the action functional for these 
four-dimensional field theories requires Weil-Petersson metrics.

\begin{remark}
For the purpose of this introduction, by a \emph{field theory} we mean a functional space $\mathcal{M}$ of
geometric data on a manifold $Y$ (such as Riemannian metrics, connections on a principal bundle on $Y$,
sections of vector bundles, \ldots ), and an action functional $ S\colon \mathcal{M}\to\mathbb{R} $ 
defined on it.
\end{remark}

In a compactification, we consider a field theory on a $D$-dimensional space-time manifold $Y=X\times \mathbb{R}^4$ which 
is the direct product of an $m$-dimensional manifold $X$ with the four-dimensional space-time $\mathbb{R}^4$. $X$ is known as 
the compactification manifold, and it is endowed with a Riemannian metric and other geometric structures that appear in the definition of 
the field theory. In the limit of small radius of $X$, there exists a four-dimensional field theory limit which can be obtained 
by imposing four-dimensional Poincar\'e invariance and
by expanding the fields around solutions of the Euler-Langrange equations associated with the action $S$.
In particular, the most general metric ansatz for a Poincar\'e invariant compactification is
$$
g_{IJ} = \left(
\begin{array}{cc}
 f \eta_{\mu\nu} & 0 \\
 0 & g_{ij} \\
\end{array} \right)
$$
where the tangent space indices are $0 \leq I < 4 + m = D$, $0 \leq \mu < 4$, and
$1 \leq i,j \leq m$. Here, $\eta_{\mu\nu}$ is the Minkowski metric, $g_{ij}$ is a metric on $X$, and 
$f$ is a real-valued function on $X$. Usually, the field theory dynamics of $g_{IJ}$ is governed by the $D$-dimensional 
Hilbert-Einstein action in general relativity. In this case, the vacuum field equations impose $g_{IJ}$ to be Ricci-flat. Given our
metric ansatz, this requires $f$ to be constant, and the metric $g_{ij}$ on $X$ to
be Ricci-flat. 

If the manifold $X$ admits a Ricci flat metric, this metric will often not be unique; rather,
there will exist a moduli space of solutions to the Ricci-flat equations. Physically, one expects
that the fields that specify the choice of the Ricci-flat metric on $X$ will vary slowly
as a function of the space-time coordinates on $\mathbb{R}^4$. General arguments suggest that these variations
must be described by variations of four-dimensional fields, governed by an \emph{Effective Field Theory} (EFT).
For simplicity, by this EFT we mean the four-dimensional field theory that emerges 
in the small radius limit of $X$, when the geometric data on $\mathbb{R}^4\times X$ restricted to 
$X$ satisfies the Euler-Lagrange equations and Poincar\'e invariance on $\mathbb{R}^4$. In this limit, the action functional of the 
EFT is defined on a functional space of geometric data on $\mathbb{R}^4$.

Given an explicit parametrization of the family of Ricci-flat metrics on $X$ (e.g. $g_{ij}(t_a)$ for
some parameters $\{ t_a \}_a$), the EFT action functional can be computed by
promoting the parameters $\{ t_a \}_a$ to $4$-dimensional fields $\{ t_a(x) \}_a$, and by determining the 4-dimensional 
action as a power series derived from the $D$-dimensional action.
For the Hilbert-Einstein action, if we expand in powers of 
$4$-dimensional derivatives, we obtain the four-dimensional action functional
\begin{eqnarray}
S^{GR}_{EFT}&=&\hspace{-0.3cm} \int_{\mathbb{R}^4 \times X} d^{(10)}\mathrm{Vol}\,\, scal(g_{IJ}) \nonumber\\
 &=& \hspace{-0.3cm} \int_{\mathbb{R}^4\times X}  d^4 x d^m y \sqrt{\det g(t)} scal(g_{ij}) + \nonumber \\
&&\hspace{-0.3cm}  \int_X d^m y \sqrt{\det g(t)} g^{ik}(t)
g^{jl}(t) \frac{\partial g_{ij}}{\partial t_a} 
\frac{\partial g_{kl}}{\partial t_b} \times \int_{\mathbb{R}^4} d^4 x \partial_\mu t_{a}(x) \partial^\mu t_{b}(x) + \ldots \label{action_GR}
\end{eqnarray}
where $y^i$ denotes a local coordinate chart on $X$, $x^\mu$ a local coordinate chart on $\mathbb{R}^4$, and 
$scal(g)$ is the scalar curvature associated with the $D$-dimensional metric. 
In general, a direct computation of Eq. \eqref{action_GR} is impossible. This is largely due to the fact
that the Ricci-flat equations cannot be explicitly solved in most of the examples of interest. 

The compactifications derived from the field theory limit of supersymmetric string theories impose further constraints on 
the four-dimensional EFT. 
In this class of theories, the field theory is defined on a manifold $Y$ of dimension $D=10$. Requiring $\calN=1$ 
supersymmetry on the four dimensional EFT and the vanishing of torsion elements fixes $X$ to be a Calabi-Yau 
three-fold. In this case, the four dimensional action functional for the $\{ t_a(x) \}_a$ fields defined in Eq. \eqref{action_GR}
depends on the Weil-Petersson metric on the moduli space of K\"ahler Ricci-flat metrics on $X$. 

These field theories contain several fields in addition to the space-time metric. For instance, the field content of a heterotic string theory (see \cite{GSW, Donag}) 
also involves a principal $E_8 \times E_8$ bundle endowed with a gauge connection $A$; here, $E_8$ denotes the Cartan's exceptional simple Lie group of dimension $248$.
In a Poincar\'e invariant compactification, the theory is defined on a principal $E_8 \times E_8$ bundle $P \to \mathbb{R}^4 \times X$. For every point $x$ on $X$, the restriction of the principal bundle $P$ to $\mathbb{R}^4 \times x$ is trivial, i.e. $P\vert_{\mathbb{R}^4 \times x\hookrightarrow \mathbb{R}^4 \times X}$ is equivalent to $E_8\times \mathbb{R}^4 $.

One can show how in the small radius limit of $X$, the corresponding field theory can be described as an effective gauge theory on $\mathbb{R}^4$ with gauge group $H$. 
The corresponding four-dimensional action can be determined by expanding the ten-dimensional Yang-Mills functional around a background reducible connection $A_0$ on 
$P \to \mathbb{R}^4 \times X$. In particular, if one considers a subgroup $G$ of $E_8$ and takes $A_0$ to be a connection on a principal $G$-subbundle of 
$P \to \mathbb{R}^4 \times X$, the gauge group $H$ of the EFT will be the commutant of $G\hookrightarrow E_8$. 
For instance, in many applications $G$ is chosen to be the special unitary group $SU(r)$, with $2<r<6$. 
In this case, the Euler-Lagrange equations associated with the Yang-Mills functional require $A_0$ to be a Hermite Yang-Mills unitary connection on $P_G\to X$.

Similarly to the K\"ahler Ricci-flat equations, if the bundle $P\to X$ admits a Hermite Yang-Mills connection, it will not be unique; rather there will be a moduli space of $E_8$ connections on $P\to X$ with $G$-holonomy. Although we do not know of a detailed description of such moduli spaces for most examples of interest, we need only to 
work with the space of local deformations around a particular $A_0$ to derive the action functional of the EFT.
In particular, in order to determine the action functional that governs the dynamics of such fields on $\mathbb{R}^4$, one has to
expand the $10$-dimensional action functional in the small radius limit of $X$, 
using small perturbations of $A_0$ that preserve the linearized Yang-Mills equations on $P\to X$ and Poincar\'e invariance on $\mathbb{R}^4$.

More precisely, given a local coordinate chart $\{ z^i,\bar{z}^{\bar{\jmath}} \}_{i,j=1}^3$  on $X$, $\{ x^\mu \}_{\mu=1}^4$ on $\mathbb{R}^4$, and a trivialization of $P$ one can expand the gauge connection $A$ around $A_0$ as
$$A (z, x) = A_{0,i} dz^i + A_{0,\bar{\jmath}} d\bar{z}^{\bar{\jmath}} + A_{\mu}(x)dx^{\mu} +
t^{\ast}_p (x)\frac{\partial  A_{\bar{\jmath}}}{\partial \bar{t}_p} d\bar{z}^{\bar{\jmath}}+ t_p (x)\frac{\partial  A_{i}}{\partial t_p} dz^i + \ldots $$
Here, $A_{\mu}dx^{\mu}$ is the $4$-dimensional $H$-gauge connection and $\{ t_p \}$ is a local coordinate chart on the space of infinitesimal deformations of the connection $A_0$ that preserve the linearized Yang-Mills equations. The ellipsis denotes higher order corrections in $t$ and, also, corrections by terms which do not preserve the linearized Yang-Mills equations. One can assume that both corrections are neglectable at low energies. If we expand the pure Yang-Mills action in $10$ dimensions assuming our Poincar\'e invariant ansatz, we find 
\begin{eqnarray}
S^{YM}_{EFT} &=& \int_{\mathbb{R}^4 \times X} d^{(10)}\mathrm{Vol}\,\, {\rm Tr}\left( F_{IJ} F^{IJ} \right) \nonumber \\
&=& \int_{X} d^{(6)}\mathrm{Vol}\,\, {\rm Tr} \left( \frac{\partial  A_{i}}{\partial t_p}  \frac{\partial  A_{\bar{\jmath}}}{\partial \bar{t}_p} \right) g^{i\bar{\jmath}} \times \int_{\mathbb{R}^4} d^4 x\, \partial_\mu t_p \partial^{\mu} t^{\ast}_{\bar{q}} + \ldots \label{action_YM}.
\end{eqnarray}
Hence, determining the effective action for the $t$ fields requires computing generalized Weil-Petersson metrics 
on the moduli space of $E_8$ connections on $P$ with $G$-holonomy, as defined in Eq. \eqref{action_YM}. The numerical methods that we introduce in this paper are useful when $G=SU(r)$ 
and the principal $SU(r)$-subbundle underlies a family of stable holomorphic vector bundles $E \to X$ (with $c_1(E)=0$, ${\rm rank}\, (E)=r$). In this case, one can use balanced embeddings to approximate the Hermite Yang-Mills connections, identify the space of infinitesimal deformations of the connection $A_0$ with sheaf cohomology groups, and approximate the Weil-Petersson metrics using natural extensions of the ideas introduced in this paper.

\subsection{Notation}

Throughout this paper, $X$ denotes a smooth projective Calabi-Yau manifold of complex dimension $n$; $\cK_X$ denotes its corresponding canonical bundle.
The corresponding holomorphic $n$-form that generates $H^{0}(X,\, \cK_X)$ is denoted by $\nu$. $\cL$ is the defining polarization, 
i.e. an ample line bundle on $X$. We denote by $\omega$ the K\"ahler two-form, with $[ \omega ] = c_1(\cL)$. 
By $ds^2=g_{i\bar{\jmath}}dz^id\bar{z}^{\bar{\jmath}}$ we mean the compatible Riemannian metric on $X$, and by $h$ the compatible Hermitian metric on $\cL$ whose curvature is $\omega$.

\section{Weil-Petersson metrics on moduli of polarized manifolds}

A \emph{holomorphic family} of compact polarized K\"ahler manifolds $(X_t,g_t)$ parametrized by $t\in\cT$ is a complex manifold
$\cX$ together with a proper holomorphic map $\pi\colon \cX \to \cT$ which is of maximal rank. This means that the differential of $\pi$ is surjective everywhere, and that $\pi^{-1}(t)$ is compact for any $t\in\cT$.

Given a base point $0\in\cT$ we say that $\pi^{-1}(t)=X_t$ is a deformation of $X_0$. Locally, $\cX$ is a trivial fiber product $\cX\vert_{\calU}\simeq \calU \times X_t $. If $T_t\cT$ denotes the holomorphic tangent space to $\cT$ at $t$, we can define the infinitesimal deformation or Kodaira-Spencer map:
$$
\rho_t\colon T_t\cT \longrightarrow H^1(X_t,\, TX_t).
$$
where $H^1(X_t,\, TX_t)$ can be identified with the harmonic representatives of (0,1) forms with values in the holomorphic tangent bundle $TX_t=T^{1,0}X_t$. 
In other words, $H^1(X_t, TX_t)\sim H^{0,1}_{\bar{\partial}}(TX_t)$. 
We know that the K\"ahler metric $g_t$ on $X_t$ induces a metric on $\Lambda^{0,1}(TX_t)$. Thus,  for $v_1,
v_2 \in  T_t\cT$, we can define a K\"ahler metric at $t\in \cT$,
\begin{equation}
G(v_1,v_2)=\int_{X_t}\langle \rho(v_1),\rho(v_2)\rangle_{g(t)} d {\rm Vol}(g(t)).\label{WP-def}
\end{equation}
G is possibly degenerate. If $\rho$ is injective and $g(s)$ satisfy an Einstein type condition, one says
that $G$ is a Weil-Petersson metric on the Kuranishi space.

\subsection{Weil-Petersson metric for Calabi-Yau's manifolds}\label{WP-CY}

Suppose now that $\cX \to \cT$ is a family of polarized Calabi-Yau manifolds $(X_t=\pi^{-1}(t),\, \cL_t)$, naturally equipped with a 
unique Ricci-flat K\"ahler metric $g_t$ in a given K\"ahler class.
We can identify the tangent space  $T_t\cT$, at $t \in \cT$, with $H^{0,1}_{\bar{\partial}}(TX_t)$.
This allows us to define the Weil-Petersson metric on $\cT$, the local moduli space of $(X,\, \cL)$, as follows.

\begin{definition}\label{WP-def-CY}
Let $v_1, \, v_2\in T_t\cT \simeq H^{0,1}_{\bar{\partial}}(TX_t)$; then the Weil-Petersson inner product is defined as
$$
\langle v_1, \,v_2 \rangle_{W.P.} := \int_{X_t} v^{i}_{1\bar{k}} \overline{v^{j}_{2\bar{l}}} g_{i\bar{\jmath}} g^{l\bar{k}}\, \mathrm{dVol}.
$$
\end{definition}
In this particular case Tian and Todorov proved the following
\begin{theorem} 
{\rm (Tian-Todorov, \cite{Tian,Todorov})} Let $\pi\colon \cX \to \cT \ni 0$ , $\pi^{-1}(0)=X_0$, be the family of X, then $\cT$ is a non-singular complex analytic space such that
$$
\dim_{\mathbb{C}} \cT = \dim_{\mathbb{C}}\, H^1(X_t,\, TX_t) = \dim_{\mathbb{C}} H^1 (X_t,\, \Omega^{n-1}),
$$
where $TX_t$ denotes the sheaf of holomorphic vector fields on $X_t$, and $\Omega^{n-1}$ the sheaf of holomorphic
$(n-1)$ forms.
\end{theorem}
There exists a one-to-one correspondence between $H^{0,1}_{\bar{\partial}}(TX_t)$ and $H^1 (X_t,\, \Omega^{n-1})$ given 
by the interior product and the global holomorphic $n$-form $\nu_t$ on $X_t$. This means that one can evaluate
the Weil-Petersson metric in terms of the standard cup product on $H^{n-1,1}(X_t)$. In particular, the K\"ahler potential for the Weil-Petersson metric
can be simply written as 
\begin{equation}
\Psi (t,\bar{t}) = (-1)^{\frac{n(n-1)}{2}} i^{n-2} \log \left( \int_{X} \nu_t \wedge \bar{\nu_t} \right).
\label{kahlerpotential}
\end{equation}
This closed-form expression can be used to evaluate the WP metric exactly. For instance, if we fix the differential structure on $X_t$ and consider 
variations of the complex structure in the holomorphic top form $\nu_t$, one can evaluate $\partial\bar{\partial} \Psi$ by computing differentials $\frac{\partial \nu_t}{\partial t_a}$, with $\frac{\partial}{\partial t_a}$ a basis for $T_t \cT= H^1 (X_t,\, \Omega^{n-1})$. Indeed, we will use this approach in the next section, 
where we introduce an algorithm to numerically evaluate $\partial\bar{\partial} \Psi$. However, a more traditional approach
to computing $\partial\bar{\partial} \Psi$ involves using the standard cup product to express $\int_{X} \nu_t \wedge \bar{\nu_t}$ as a function of $t$. 
This approach was first used by \cite{Can} in the context of Calabi-Yau compactifications, as we show in the following subsection.

\subsection{Example: the Quintic in $\mathbb{P}^4$} 
In this paper we will study different K\"ahler metrics on a family of quintic 
hypersurfaces $X=Q$ in $\mathbb{P}^4$. The particular representation of this family of Calabi-Yau three-folds is very convenient
for the study of the analytical and geometrical objects that interest us here. The Hodge numbers of this class of three-folds are
$h^{1,1}=1$ and $\dim_{\mathbb{C}} H^1 (Q,\, \Omega^{2})=h^{2,\,1}=101$. Additionally, one
can describe its moduli space explicitly. If we define
$$
W = \{\, P\, \vert \, P\,\, {\rm a\,\, homogeneous\,\, quintic\,\, polynomial\,\, of } \,\, Z_0, Z_1, Z_2, Z_3, Z_4 \},
$$
one can verify that $\dim W = 126$. Hence, the parameter space $t \in \mathbb{P} W = \mathbb{P}^{125}$ consists of
a family of quintic hyper-surfaces in $\mathbb{P}^4$ defined by the zero-loci of the quintic polynomials $t\in\mathbb{P}^{125}$. 
As two hypersurfaces that differ by an element in ${\rm Aut}(\mathbb{P}^4 )$
are equivalent, and there exists a divisor $\mathcal{D}$ in $\mathbb{P}^{125}$ characterizing
the singular hypersurfaces in $\mathbb{P}^4$, the moduli space of smooth quintics $Q$ is given by
$$
\mathcal{M} = \left( \mathbb{P}^{125} \backslash \mathcal{D} \right) / {\rm Aut}(\mathbb{P}^4 ).
$$
The dimension of the moduli space is 101, as expected.

For simplicity, in this paper we will study the one-dimensional subspace of complex deformations of the quintic defined by
the family of polynomials
\begin{equation}
P_t(Z)=Z_{0}^5+Z_{1}^5+Z_{2}^5+Z_{3}^5+Z_{4}^5 - 5t Z_{0} Z_{1} Z_{2} Z_{3} Z_{4},
\label{quintic}
\end{equation}As $t$ and $t \exp (2\sqrt{-1}\pi l/5)$, for any $l\in \mathbb{Z}$, represent the same variety, the fundamental subset of the $t$-plane that
parametrizes the family can be chosen to be $\{\, t\, \vert\, 0\leq \arg(t) <  2\pi /5 \,\, {\rm and}\,\, t\neq 1 \}$. If $t$ is a fifth root of unity 
($t=\exp (2\sqrt{-1}\pi l /5)$ for any $l\in \mathbb{Z}$), the associated quintic is not smooth and develops double-point singularities.
\medskip

\noindent \textbf{The Weil-Petersson metric on the family of Quintics.} Candelas et. al. \cite{Can} computed 
the K\"ahler potential defined in Eq.\eqref{kahlerpotential} by evaluating the volume 
$\int_{X} \nu_t \wedge \bar{\nu_t}$ by means of cup products. 
In particular, they constructed 
a symplectic basis of 3-cycles $(A^a,\,B_b)$ in $H_3(Q_t,\,\mathbb{Z})$ for the family of quintics defined in Eq. \eqref{quintic}. The intersection numbers
of the elements of this basis of
3-cycles are
$$
A^a\cap B_b = \delta^a_b,\quad\, A^a\cap A^b = 0,\quad\,B_a\cap B_b = 0.
$$
Also, they considered the dual basis of three-forms $(\alpha_a,\,\beta^b)$ in $H^3(Q_t,\,\mathbb{Z})$
$$
\int_{A^a} \alpha_b = \delta^a_b,\quad\, \int_{B_a} \beta^b = \delta_a^b,
$$
with the other integrals vanishing. The wedge product of the elements of the basis of three-forms satisfy the integrals
$$
\int_Q \alpha_a\wedge\beta^b = \delta_a^b, \quad\, \int_Q \alpha_a\wedge\alpha_b=\int_Q \beta^a\wedge\beta^b=0.
$$
In this case, the holomorphic three-form $\nu_t$ can be written using the basis of three-forms as
$$
\nu_t = z^a\alpha_a - \mathcal{G}_b\beta^b,
$$
and the volume of $Q_t$ can be written as
$$
\int_{Q} \nu_t \wedge \bar{\nu_t} = \bar{z}^a \mathcal{G}_a - z^a\overline{\mathcal{G}}_a.
$$
Using this expression, the Weil-Petersson metric is simply
\begin{equation}
g_{t\bar{t}}=-i\partial_t \overline{\partial}_t \log \left(\bar{z}^a \mathcal{G}_a - z^a\overline{\mathcal{G}}_a \right).
\label{cande}
\end{equation}
\begin{figure}
\begin{center}
    \includegraphics[width=5in]{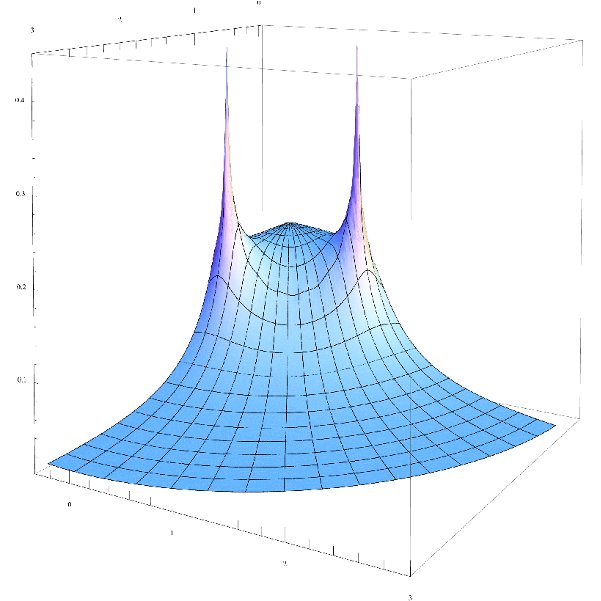}
\end{center}
  \caption{Weil-Petersson metric (vertical axis) on the $t$-plane (horizontal plane) of 1-dimensional moduli of Calabi-Yau Quintic 3-folds,  $Z_{0}^5+Z_{1}^5+Z_{2}^5+Z_{3}^5+Z_{4}^5 - 5t Z_{0} Z_{1} Z_{2} Z_{3} Z_{4}$.}
    \label{candelas}
\end{figure}
Hence, in order to determine the Weil-Petersson metric on the family of quintics, 
it is sufficient to evaluate the periods $z^a =\int_{A^a} \nu_t$ and $\mathcal{G}_b= \int_{B_b}\nu_t$. To that end, we consider the vector space  $V_j$ generated by the vectors 
$$\left(\begin{matrix}
\frac{\partial^k}{\partial t^k}z^a(t)\\
\frac{\partial^k}{\partial t^k}\mathcal{G}_b(t)
\end{matrix}\right),$$
for $0\leq k\leq j$. For generic values of $t$  
the dimension of $V_j$ must be constant; in our case,  $\dim V_j$ cannot be larger than $4$. Thus, by expressing any element of the over-complete basis of
$V_5$ as a function of the remaining elements that form a complete basis of $V_5$, we obtain a non-trivial ordinary differential equation relating different periods. 
This equation is known as the Picard-Fuchs equation, and we refer to \cite{Mor} for more mathematical details on this topic. The form of these equations depends on the particular choice of local coordinates on the space of deformations and on the choice of $\nu_t$. 
In particular, the solutions of the Picard-Fuchs equations may exhibit a singular behavior. However, the types of singularities that can occur here are well characterized. 
In the case of the family of quintics, the Picard-Fuchs equations are generalized hypergeometric equations and can be solved by expanding the 
integrands of the periods in powers of $t$. Each coefficient of the power series leads to an integral that can be evaluated using residue formulas. 
The periods are then extended using analytic continuation to the fundamental domains $|t|<1$ with $0<\arg(t)<\frac{2\pi}{5}$, and $|t|>1$ with $0<\arg(t)<\frac{2\pi}{5}$. 
We have written two simple programs in Mathematica and Maple, for the family of quintics in Eq. \eqref{quintic}, that numerically integrate the power series that define the periods. We then used these integrals to compute the WP metric. In Fig. \ref{candelas}, we show our evaluation of Eq. \eqref{cande} for $0< \vert t \vert \leq 3 $ and $0\leq \arg(t) <  2\pi /5$.  
Although the periods can be integrated exactly at the singular points ($t=1,\infty$), it is difficult to obtain simple formulas
in the more general case (e.g. the Calabi-Yau manifold is not a hypersurface) . For this reason, we explore a different approach to computing the WP metric in the next section.

\section{Numerical evaluation of the Weil-Petersson metrics using deformations of the holomorphic $n$-form}
\subsection{Description of the method}
In this section we describe how to compute Weil-Petersson metrics by taking infinitesimal deformations of $\nu_t$ associated with changes 
of the complex structure on $X$. First, we make an important remark regarding the notation. By $X$ we denote a Calabi-Yau differentiable manifold with no complex structure defined on it. $X_{t}$ denotes the same differentiable manifold endowed with an integrable complex structure parametrized by $t$. We denote by $U\subset X$ an open subset of the differentiable manifold $X$, such that $U\subset X$ is independent of any complex structure one defines on $X$.

Every tangent vector on the moduli space of complex structures, $v^a \frac{\partial}{\partial t_a} \in T_{t_0}\cT$, yields an infinitesimal deformation of the complex structure on $X_{t_0}$. In particular,
in any local coordinate patch on $U\subset X$ we can relate the holomorphic coordinates on $X_{t_0}$ with the holomorphic ones
on $X_{t_0+v}$ by introducing an appropriate infinitesimal diffeomorphism. Let $\{ w^i \}_{i=1}^n$ be a local holomorphic coordinate system for $X_{t_0}$ on $U\subset X$, and $\{ y^i \}_{i=1}^n$ be a local holomorphic coordinate system for $X_{t_0+v}$ on the same subset $U$.  Therefore, on $U$, we can relate the $w$-coordinates and the $y$-coordinates as:
\begin{equation}
y^i = w^i + v^{a}\vartheta_a^{i}(w,\bar{w}) + O(v^2),
\label{diffeo}
\end{equation}
with $\vartheta$ a non-holomorphic section of $T^{1,0}X_{t_0}$, and $\{ \frac{\partial}{\partial t_a} \}_a$ a basis for $T_{t_0}\cT$.

Hence, using the $w$-coordinate system, we can write the holomorphic top form $\nu_{t_0 + v}$ on $X_{t_0 + v}$
as a non-holomorphic $n$-form in $\Omega^{n,0}(X_{t_0})\oplus \Omega^{n-1,1}(X_{t_0})$. In particular,
\begin{equation}
\nu_{t_0 + v} = \nu_{t_0} + v^{a}\partial_{t_a} \nu_{t_0} + O(v^2),
\label{threeform}
\end{equation}
where the $O(v^2)$ terms are irrelevant for the purpose of evaluating the WP metric. The term $\partial_{t_a} \nu_{t_0}$ is computed as the pull-back of the infinitesimal diffeomorphism defined by $\vartheta_a^{i}(w,\bar{w})$. Thus, given a basis of deformations $\partial_{t_a} \nu_{t_0}\in \Omega^{n,0}(X_{t_0})\oplus \Omega^{n-1,1}(X_{t_0})$ and vectors $v_1,\, v_2\in T_{t_0}\cT$, we can write the Weil-Petersson inner product as
\begin{equation}
\langle v_1,\, v_2\rangle_{WP} = -\frac{v_1^{a} \overline{v_2^{b}}\int_X \partial_{t_a} \nu_{t_0} \wedge \overline{\partial_{t_b} \nu_{t_0}} }{\int_X \nu_{t_0}\wedge \overline{\nu_{t_0}}} + \frac{v_1^{a}\overline{v_2^{b}}  \int_X \partial_{t_a} \nu_{t_0} \wedge \overline{\nu_{t_0}} \int_X \nu_{t_0}\wedge \overline{\partial_{t_b} \nu_{t_0}} }{\left( \int_X \nu_{t_0}\wedge \overline{\nu_{t_0}} \right)^2}, \label{WPexact}
\end{equation}
where we have expanded the K\"ahler potential in Eq. \eqref{kahlerpotential} using the deformations of $n$-forms defined by Eq. \eqref{threeform}.
Therefore, evaluating Eq. \eqref{WPexact} requires us to solve two main technical problems:
\begin{itemize}
\item A choice of $\vartheta_a^{i}(w,\bar{w})$, which is not unique and depends on the particular geometry of the Calabi-Yau manifold.
\item Computing several integrals on $X$.
\end{itemize}
In the remainder of this subsection we derive closed-form expressions for $\vartheta_a^{i}(w,\bar{w})$ and $\partial_{t_a} \nu_{t_0} $ in the case of 
complete intersection Calabi-Yau manifolds.
In subsection 3.2 we introduce a Monte-Carlo approach for evaluating integrals on $X$. We exhibit the resulting Weil-Petersson metric on the one-parameter family of quintics in subsection 3.3.

If $X_{t_0}$ is a complete intersection Calabi Yau manifold, there exists a natural choice for $\vartheta_a^{i}(w,\bar{w})$ that we describe as follows. 
Let $\{ P_\alpha(Z) \}_{\alpha=1}^{m-n}$ be a basis of homogeneous polynomials in $\mathbb{P}^m$ whose common
zero loci define $X_{t_0}$. Let us suppose that given two independent infinitesimal deformations of the complex structure, $v_1, \, v_2\in T_{t_0}\cT \simeq H^{0,1}_{\bar{\partial}}(TX_{t_0})$, we can find two sets of polynomials,
$\{P_\alpha(Z)+ v_1^a  \partial_{t_a} P_\alpha(Z) \}_{\alpha=1}^{m-n}$ and $\{ P_\alpha(Z)+v_2^a  \partial_{t_a} P_\alpha(Z) \}_{\alpha=1}^{m-n}$, that
parametrize isomorphic deformations of the complex structure. We set a coordinate atlas on $X\subset\mathbb{P}^m$ by choosing inhomogeneous local coordinates $\{ w_i=Z_i/Z_0\}_{i=1}^m$ on $\mathbb{P}^m$, $n$ coordinates as local coordinates on $U\subset X$, and the remaining $n-m$ coordinates as dependent of the $n$ coordinates on $U\subset X\subset \mathbb{P}^m$. In other words, for any point $x\in X$, by making a unitary change of coordinates on $\mathbb{P}^m$ we can always set $\{w_i \}_{i=1}^n$ to be a local coordinate system on an open subset of $X_{t_0}$ that contains $x$, while the remaining coordinates $\{w_i = w_i(w_1,\ldots,x_n) \}_{i=n+1}^m$ on $X_{t_0}\subset\mathbb{P}^m$ can be expressed as a function of $\{w_i \}_{i=1}^n$. We write the defining polynomials in inhomogenous coordinates, as
\begin{eqnarray*}
p_\alpha(w) =p_\alpha(w,t=t_0)&=& P_\alpha(Z,t=t_0)/Z_0^{\deg P_\alpha},\\
\partial_{t_a} p_\alpha(w)=\partial_{t_a} p_\alpha(w,t)\vert_{t=t_0} &=& \partial_{t_a}  (P_\alpha(Z,t)/Z_0^{\deg P_\alpha}) \vert_{t=t_0},
\end{eqnarray*}
where $\deg P_\alpha\in\mathbb{Z}^+$ is the degree of the homogeneous polynomial $P_\alpha$. If $\vartheta_a^{i}(w,\bar{w})$ are vector 
fields on $X_{t_0}\subset\mathbb{P}^m$ corresponding to the deformations $\{  \partial_{t_a}  p_\alpha(w) \}$, and
$\{y_i \}_{i=1}^{m}$ is a holomorphic local coordinate system on $X_{t_0+v}\subset\mathbb{P}^m$,
the following equation holds for an infinitesimal variation $v^b \frac{\partial}{\partial t_b}$ on the moduli space:
\begin{equation}
p_\alpha(y)+v^b  \partial_{t_b}  p_\alpha(y) = 0 = p_\alpha(w)+ v^b  \frac{\partial p_\alpha(w)}{\partial w_i} \vartheta_b^{i}(w,\bar{w}) 
+ v^b   \partial_{t_b}  p_\alpha(w) + O(v^2). \label{def_pol}
\end{equation}
Here, we have used the formula for the change of $y$ and $w$ coordinates defined by Eq. \eqref{diffeo}. Now, we can introduce a particular 
expression for the vector field $\vartheta_a^{i}(w,\bar{w})$.
\begin{proposition}\label{prop1}
Let $G_{i\bar{\jmath}}$ be a Fubini-Study metric on $\mathbb{P}^m$. Let $H_{\alpha\bar{\beta}}$ be the elements
$$
H_{\alpha\bar{\beta}}=G^{i\bar{\jmath}}\frac{\partial p_\alpha(w)}{\partial w_i}\frac{\partial \bar{p}_{\bar{\beta}}(\bar{w})}{\overline{\partial} \overline{w}_{\bar{\jmath}}}.
$$
Then, a natural choice for $\vartheta_a^{i}(w,\bar{w})$ is
\begin{equation}
\vartheta_a^{i}(w,\bar{w}) = - \left( H^{-1} \right)^{\bar{\beta}\gamma}G^{i\bar{\jmath}} \frac{\partial \bar{p}_{\bar{\beta}}(\bar{w})}{\overline{\partial} \overline{w}_{\bar{\jmath}}} \partial_{t_a} p_\gamma(w). \label{vectorfield}
\end{equation}
\end{proposition}
The proof is straightforward by substituting $\vartheta_a^{i}(w,\bar{w})$ into Eq. \eqref{def_pol} with $v^b\frac{\partial}{\partial t_b}=\frac{\partial}{\partial t_a}$, 
and by using $p_\alpha(w)=0$, as $w$ lies on $X_{t_0}\subset\mathbb{P}^m$.

By combining Eq. \eqref{vectorfield} for $\vartheta_a^{i}(w,\bar{w})$ and Eq. \eqref{def_pol},
we can calculate the deformation of $\nu_t$ under the infinitesimal diffeomorphism defined in Eq. \eqref{diffeo}. 
In particular, if 
\begin{equation}
\nu_{t_0 +v}=N_{i_1,\ldots , i_n}(y)dy^{i_1}\wedge\ldots\wedge dy^{i_n},\label{form2}
\end{equation}
is the holomorphic $n$-form on $X_{t_0+v}\subset\mathbb{P}^m$, $y^i = w^i + v^{a}\vartheta_a^{i}(w,\bar{w}) + O(v^2)$,
and 
\begin{equation}
dy^i = dw^i + v^{a}\frac{\partial \vartheta_a^{i}(w,\bar{w})}{\partial w^j}dw^j + v^{a}\frac{\partial \vartheta_a^{i}(w,\bar{w})}{\bar{\partial} \bar{w}^{\bar{\jmath}}}d\bar{w}^{\bar{\jmath}} + O(v^2),\label{jacobian}
\end{equation}
we can expand Eq. \eqref{form2} as in Eq. \eqref{threeform}, and determine the term 
$\frac{\partial \nu_{t}}{\partial {t_a}}(t_0)$ that we need to evaluate the Weil-Petersson metric by means of Eq. \eqref{WPexact}.

%Although one could compute Eq. \eqref{threeform} in the general case, 
For simplicity, we only provide an explicit expression of $\partial_{t_a} \nu_{t}(t_0)$ in the case $m-n=1$, i.e. $X_{t_0}$ is a 
hypersurface defined as the zero locus of a single polynomial $p(w)$. By the adjunction formula we know that $\nu_{t_0 + v}$ is the pull-back of a 
meromorphic $n$-form on $\mathbb{P}^{n+1}$ that obeys the formula
\begin{equation}
\prod_{i=1}^{n+1}dy^i = d\left( p(y)+v^a \partial_{t_a} p(y) \right)\wedge \nu_{t_0 + v}. \label{residue}
\end{equation}
Using equations \eqref{threeform} and \eqref{residue} and the transformation of $dy^i$ specified in Eq. \eqref{jacobian},  
one can compute $\partial_{t_a} \nu_{t_0}$ as

\begin{align}
\partial_{t_a} \nu_{t_0}&=-\frac{1}{\frac{\partial p}{\partial w^{n+1}}}\left(\sum_{i=1}^{n+1} \frac{\partial \vartheta_a^{i}(w,\bar{w}) }{ \partial w^i} \right) \prod_{i=1}^{n}dw^i \label{diffform} \\
-  &\sum_{i,j=1}^{n}  \frac{(-1)^{n-i}}{\frac{\partial p}{\partial w^{n+1}}}  \left( \frac{\partial \vartheta_a^{i}(w,\bar{w})}{\partial \bar{w}^{\bar{\jmath}}} +  \frac{\partial \bar{w}^{\overline{n+1}}}{\partial \bar{w}^{\bar{\jmath}}} \frac{\partial \vartheta_a^{i}(w,\bar{w})}{\partial \bar{w}^{\overline{n+1}}} \right)dw^1 \ldots \widehat{dw^i} \ldots  dw^n d\bar{w}^{\bar{\jmath}}  \nonumber \\
+ & \ldots \nonumber 
\end{align}
Here, the differentials $dw^i$ obey the Grassmann algebra of forms, $\widehat{dw^i}$ denotes
the omission of $dw^i$, and the final ellipsis denotes further terms which do not contribute to Eq. \eqref{WPexact}. 
Now, equipped with a local formula for the integrands in Eq. \eqref{WPexact}, 
we need a numerical method to evaluate the integrals that appear therein.

\newpage
\subsection{Monte Carlo integration on varieties}
\subsubsection{Importance Sampling}
In order to compute the Weil-Petersson metric using Eq. \eqref{WPexact}, one needs to evaluate multiple integrals of the type
\begin{equation}
\int_X f \nu\wedge\overline{\nu}.\label{integral}
\end{equation}
We can approximate such integrals using a standard Monte Carlo method known as Importance Sampling (IS).
In particular, as generating random samples of points on $X$ under the measure $\nu\wedge\overline{\nu}$ is difficult, one uses
instead an auxiliary measure $d\mu$. It is assumed that one can easily generate
random point sets $\{ q_l \in X \}_{1\leq l\leq N_{points}}$ on $X$ uniformly distributed under $d\mu$. 
Hence, by defining the mass function $m(x) = \nu\wedge\overline{\nu} /d\mu (x)$, we can approximate Eq. \eqref{integral} as
\begin{equation}
\int_X f \nu\wedge\overline{\nu} \simeq \frac{\mathrm{Vol}(X)}{\sum_{l=1}^{N_{points}} m_l} \sum_{l=1}^{N_{points}} f(q_l) m_l + O\left( N_{points}^{-1/2}\right), \label{montecarlos}
\end{equation}
where $N_{points}$ is the number of points used, $m_l=m(q_l)$ and $O(N_{points}^{-1/2})$ is the asymptotic decay of the standard error of the integral for large $N_{points}$.

In the case of a polarized manifold with a very ample line bundle $\cL$ we generate the point set and the auxiliary measure $d\mu$ using the Kodaira embedding 
$\imath\colon X\hookrightarrow \mathbb{P} H^{0}(X,\,\cL)^\ast$. In particular, we endow the ambient projective space $\mathbb{P} H^{0}(X,\,\cL)^\ast$ with a Fubini-Study metric $\omega_{FS}$
and take random sections $\sigma$ in $\mathbb{P} H^{0}(X,\,\cL)^\ast$ with respect to the compatible volume form. 
The zero locus of such random sections $\sigma$ are divisors with associated zero currents $T_\sigma$. One can show \cite{SZ} that the expected zero current is:
$$
E(T_\sigma) = \imath ^\ast \omega_{FS}.
$$
This implies that the expected common zero loci associated with $n$ independent random sections in $\mathbb{P} H^{0}(X,\,\cL)^\ast$ consists of $\int_X c_1(\cL)^n$ points on $X$ uniformly distributed under
$$
E(T_{\sigma_1\ldots\sigma_n}) = \frac{(\imath ^\ast \omega_{FS})^n}{n!}.
$$
Therefore, in our application of IS to integrate functions on Calabi-Yau manifolds, we consider $d\mu= (\imath ^\ast \omega_{FS})^n/n!$ to be the auxiliary measure. Also,
we generate samples of random points  by taking the common zero loci of $n$ independent random sections. The associated mass function is then 
\begin{equation}
m(x)=n!\frac{\nu\wedge\overline{\nu}}{(\imath ^\ast \omega_{FS})^n}(x).
\label{massformula}
\end{equation}
%\small skip

\noindent{\textbf{Example.}} As an application of this method we consider the elliptic curve $E$ in $\mathbb{P}^2$ defined as
the zero locus of the Weierstrass cubic polynomial
$$
Z_2^2 Z_0 = 4Z_1^3-60G_4 (i)Z_1 Z_0^2.
$$
Here, $G_4(i)$ is the Eisenstein series of
index 4 evaluated at the complex parameter $\tau = i$ ($G_4(i)\simeq -3.151212\ldots$).
\begin{figure}
\begin{center}
    \includegraphics[width=3in]{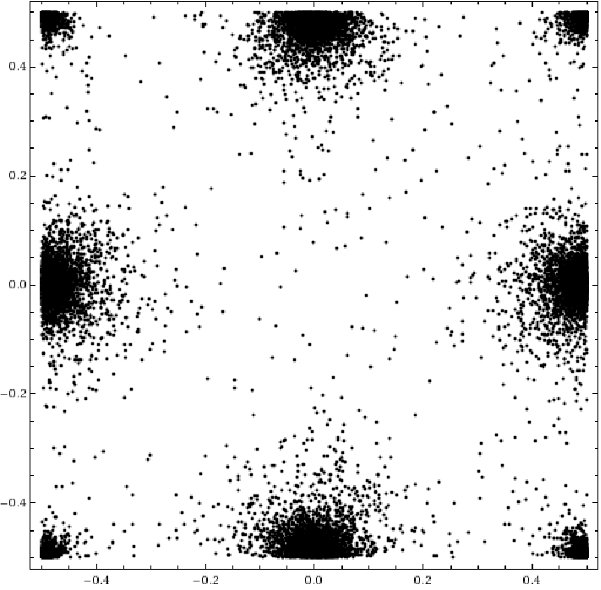}
\end{center}
  \caption{Distribution of random points on the Weierstrass cubic $Z_2^2 Z_0 = 4Z_1^3-60G_4 (i)Z_1 Z_0^2 $ resulting from the intersection with random $\mathbb{P}^1$'s 
on $\mathbb{P}^2$. Here, $\mathbb{P}^2$ is endowed with the Fubini-Study metric defined by the K\"ahler potential $\log(1+\vert Z_1/Z_0 \vert^2+\vert Z_2/Z_0 \vert^2)$.}
    \label{montecarlo}
\end{figure}
This elliptic curve can also be represented as the square torus $\mathbb{C}/\mathbb{Z}^2$ embedded
in $\mathbb{P}^2$. The Calabi-Yau area $2$-form corresponds to the flat area form
inherited from the complex plane $\mathbb{C}$, in the quotient $\mathbb{C}/\mathbb{Z}^2$. Intersections of divisors associated with random sections
in $\mathbb{P}^2 = \mathbb{P}H^{0}(E, \mathcal{O}(3pts))$ with the elliptic curve are equivalent to intersections of 
the cubic $E\hookrightarrow\mathbb{P}^2$ with random projective lines $\mathbb{P}^1\hookrightarrow\mathbb{P}^2$. 
The Fubini-Study area two-form on $E$ yields the particular distribution of points shown in Fig. \ref{montecarlo}. 
As we have derived a closed-form expression for $\nu\wedge\bar{\nu}/\imath ^\ast \omega_{FS}$, we can use Importance Sampling on $E$ to 
evaluate integrals with area two-form $\nu\wedge\bar{\nu}$
by means of random point sets associated with these random sections.

\noindent{ $ $}

\subsubsection{Refinements of Importance Sampling}
In many applications the most efficient method to evaluate integrals is the standard Importance Sampling method applied to projective manifolds, as we have just 
described. 
This strategy requires a single optimal Fubini-Study  metric and a point set generated by intersecting divisors associated with independent random sections. 
The optimal Fubini-Study metric could be the $\nu$-balanced metric \cite{Don} (see Eq. \eqref{tamp} in section below), which aims to approximate
the K\"ahler Ricci flat metric. Once the Fubini-Study metric has been fixed, 
the minimum number of points used in the Monte Carlo sample can be adjusted to attain any given degree of accuracy and precision in the value of the integral.

If the ratio of the maximum over the minimum of the mass function defined in Eq. \eqref{massformula} is very large,
the convergence of the IS method will likely be worse. In this case, one can increase
the number of points to approximate the integrals with more accuracy and precision. 
However, if we integrate functions with extremely slow evaluation maps (slow from the computational point of view)
we will not be able to use a high number of points to reduce the error. 
A different strategy that we explore in this subsection 
involves improving the distribution of the point set while keeping the number of points constant.
In particular, one can use an optimal combination of several Fubini-Study metrics and subsets on $X$ to generate an Improved 
Point Set. 
\begin{remark}
By an \emph{Improved Point Set} (IPS), we mean a distribution of points on $X$ whose associated mass formula $m_{IPS}(x)$
obeys 
$$
\frac{{\rm max} (m_{IPS}(x))}{{\rm min} (m_{IPS}(x))} \ll \frac{{\rm max} (m(x))}{ {\rm min} (m(x))}.
$$
\end{remark}

A natural strategy to construct such an IPS is to consider several Fubini-Study metrics $\{ \omega^q_{FS} \}_{q=1}^{Q}$ and the associated mass function
\begin{eqnarray}
m_{IPS}(x) &=& n!\frac{\nu\wedge\overline{\nu}}{(\imath ^\ast \omega^q_{FS})^n}(x) \text{ for } q \in [1,\, Q]
\nonumber \\
&&\text{ such that }\left\vert n!\frac{\nu\wedge\overline{\nu}}{(\imath ^\ast \omega^q_{FS})^n}(x) - 1 \right\vert = \min_{1\leq j \leq Q} \left\vert n!\frac{\nu\wedge\overline{\nu}}{(\imath ^\ast \omega^j_{FS})^n}(x) - 1 \right\vert. \label{IPSmassformula}
\end{eqnarray}
Here, for simplicity we normalized volumes such that $n!\int_X \nu\wedge\overline{\nu} = \int_X (\imath ^\ast \omega^q_{FS})^n$. The mass function in Eq. \eqref{IPSmassformula} 
implies that we can decompose $X$ as a disjoint
union of open subsets $X=\coprod_{q=1}^Q U_q$ with non-zero volume. In other words, each Fubini-Study metric 
$\omega^r_{FS}\in \{ \omega^q_{FS} \}_{q=1}^{Q}$ is associated with the subset $U_r\subset X$ defined by
$$
x\in U_r \,\,{\rm if}\,\, m_{IPS}(x) = n!\frac{\nu\wedge\overline{\nu}}{(\imath ^\ast \omega^r_{FS})^n}(x).
$$
In this refinement of the Importance Sampling method one has to generate point sets by considering only random points that are located in appropriate
subsets. In particular, sets of $n$ independent random sections with respect to the Fubini-Study volume form $\left( \omega^q_{FS} \right)^n$ will yield random points on 
the whole manifold $X$; however, only those points that lie on $U_q\subset X$ can be accepted. 
If a point $y\notin U_q $ is generated as the common zero locus of $n$ independent random sections under the $q^{th}$-measure, it will not be included in the point set. 
This means that there exists another subset $U_r$ and a metric $\omega^r_{FS}$, with $y\in U_r$, such that $n$-tuples of random sections under the $r^{th}$-measure 
will generate points on $U_r$ that are more similarly distributed to the reference volume form $\nu\wedge\overline{\nu}$. 

  \begin{figure}
  \begin{center}
  \begin{tabular}{ccc}
      \includegraphics[width=1.5in]{adaptative_1chambers_cubic.jpg}&
      \includegraphics[width=1.5in]{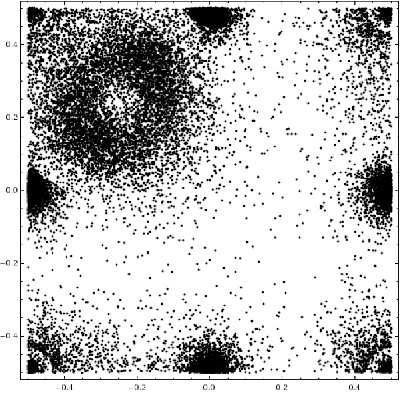}&
      \includegraphics[width=1.5in]{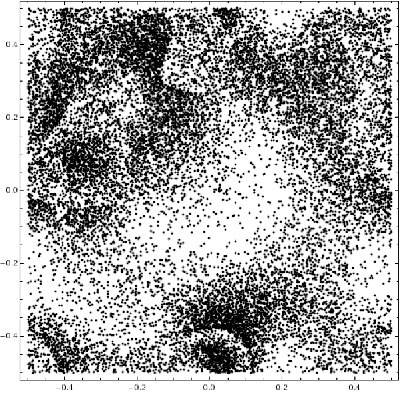}\\ \\
      \includegraphics[width=1.5in]{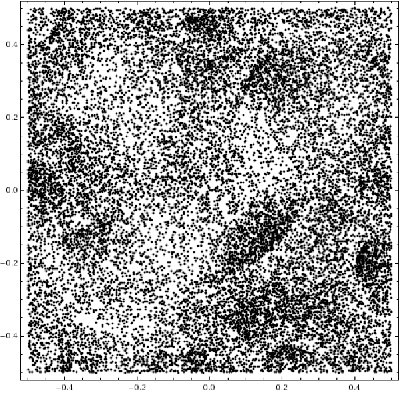}&
      \includegraphics[width=1.5in]{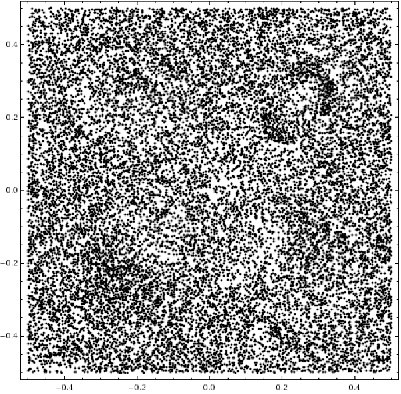}&
      \includegraphics[width=1.5in]{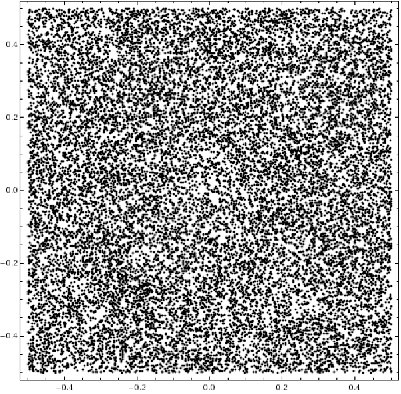}
  \end{tabular}
  \end{center}
    \caption{Distribution of 20,000 random points on the Weierstrass cubic, for 
  1, 2, 3, 5, 11, and  19 optimally chosen Fubini-Study metrics and subsets.}
      \label{refined_montecarlo}
  \end{figure}

There is not a unique answer to the question of how to generate optimal sets of Fubini-Study metrics and subsets on $X$ when $Q>1$.
Here, we describe the particular solution that we used in this paper. Before we discuss the details
of our approach to generate these IPS, we need a definition.

\begin{definition}
Given any point $x\in X$, there exists an $x$-\emph{mass-one} Fubini-Study metric $\omega_{FS}(\lambda_x)$
on $\mathbb{P} H^{0}(X,\,\cL)^\ast$ that satisfies
$$
n!\frac{\nu\wedge\overline{\nu}}{\imath ^\ast \omega_{FS}(\lambda_x)^n}(x) = 1.
$$
\end{definition}
One can always construct this metric $\omega_{FS}(\lambda_x)$ as follows. Let us consider an orthonormal basis $\{ s_\alpha \}_{\alpha=1}^{N+1}$ for $H^{0}(X,\,\cL)$ with respect to the $\nu$-balanced Fubini-Study metric. Now, in this basis, we introduce the matrix $\lambda_x$ 
$$
\lambda_x=\frac{1}{1+\epsilon}\left( \mathbf{1} + \epsilon P_x \right),
$$
with $\mathbf{1}$ the identity matrix, $P_x=P_x^2$ the projector on the ray generated by $x\mapsto \mathbb{P} H^{0}(X,\,\cL)^\ast$, and  
$$
\epsilon=\left( n!\frac{\nu\wedge\overline{\nu}}{\imath ^\ast \omega_{FS}(\mathbf{1})^n} (x) \right)^{\frac{1}{n}}-1.
$$
It is then straightforward to show that if 
$
\log \left( \sum_{\alpha\beta} \left(\lambda^{-1}_x\right) ^{\bar{\beta}\alpha}s_\alpha\bar{s}_{\bar{\beta}} \right)
$
is the K\"ahler potential for $\omega_{FS}(\lambda_x)$, then
$$
n!\frac{\nu\wedge\overline{\nu}}{\imath ^\ast \omega_{FS}(\lambda_x)^n}(x) = 1.
$$

We can generate optimal sets of Fubini-Study metrics by using iteratively the refined mass formula in Eq. \eqref{IPSmassformula} to further incorporate
optimally chosen $x$-\emph{mass-one} metrics. In particular, let us consider a set of Fubini-Study metrics $\{ \omega^q_{FS} \}_{q=1}^{Q}$ with $Q>0$
that we want to refine further by adding two extra metrics. In this case, we will search for the absolute maximum $x_{max}$ and minimum $x_{min}$ of the mass function $m_{IPS}(x,\{ \omega^q_{FS} \}_{q=1}^{Q})$,
and add the metrics $\omega_{FS}(\lambda_{x_{min}})$ and $\omega_{FS}(\lambda_{x_{max}})$ to the set
$$
\{ \omega^q_{FS} \}_{q=1}^{Q+2}\,\,\text{such that} \,\,  \omega^{Q+1}_{FS}=\omega_{FS}(\lambda_{x_{max}}),\,\,
\omega^{Q+2}_{FS}=\omega_{FS}(\lambda_{x_{min}}).
$$
Fig. \ref{refined_montecarlo} and Fig. \ref{histograms} show a few examples of Improved Point Sets on the Weierstrass
cubic defined above, and on a quintic three-fold. It is important to remark that this refinement of the Importance Sample method on Calabi-Yau manifolds 
is of independent interest. 
\smallskip

\begin{figure}
\begin{center}
\begin{tabular}{cc}
    \includegraphics[width=2.2in]{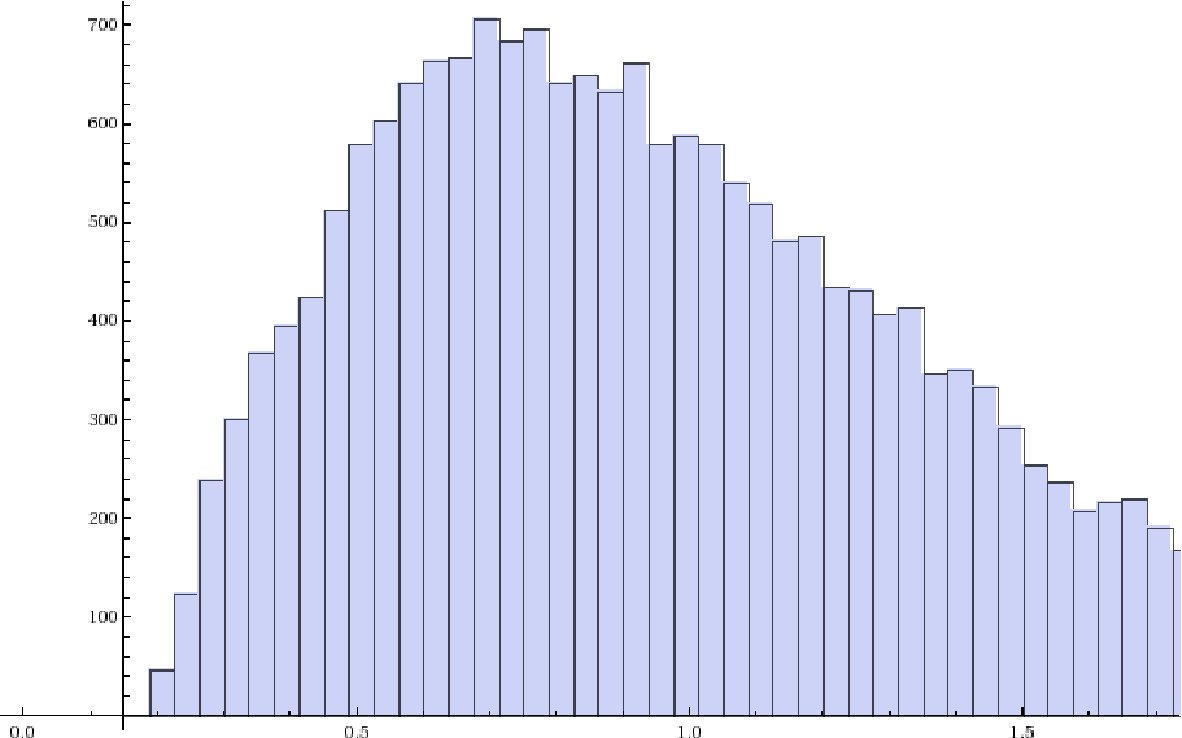} & 
    \includegraphics[width=2.2in]{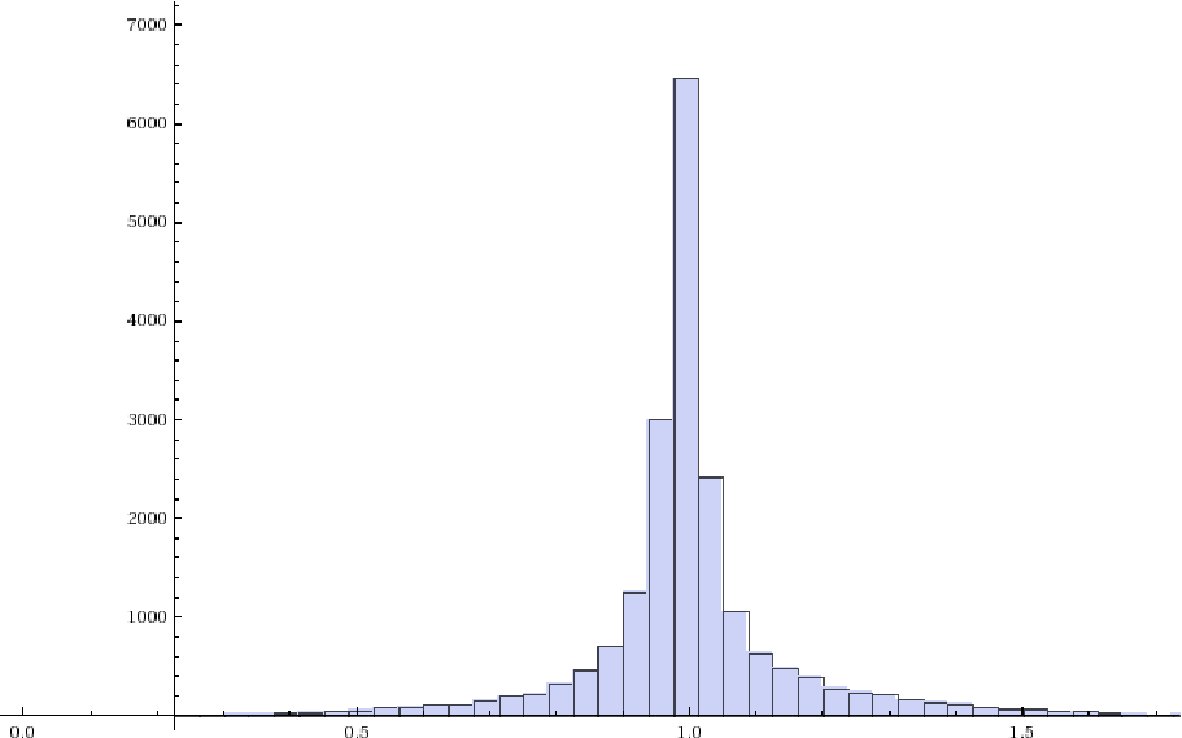} \\
\end{tabular}
\end{center}
  \caption{Distribution of masses for points on the Calabi-Yau quintic $Z_{0}^5+Z_{1}^5+Z_{2}^5+Z_{3}^5+Z_{4}^5 - 0.246\times Z_{0} Z_{1} Z_{2} Z_{3} Z_{4}$, using 1
  Fubini-Study metric (left) and 19 optimally chosen Fubini-Study metrics and subsets (right).}
    \label{histograms}
\end{figure}

\subsection{Example: the family of quintics}\label{method1}

Equipped with Eq. \eqref{WPexact}, Eq. \eqref{diffform} and the Monte Carlo method just described, 
one can evaluate Weil-Petersson metrics for a large class of families using deformations of holomorphic $n$-forms. 
For instance, one can compute the WP metric on the one-parameter family of quintic three-folds introduced in section 2.2,
\begin{equation}
P(Z)=Z_{0}^5+Z_{1}^5+Z_{2}^5+Z_{3}^5+Z_{4}^5 - 5t Z_{0} Z_{1} Z_{2} Z_{3} Z_{4}.\label{poly}
\end{equation}
By generating 2,000,000 points on the Fermat quintic ($t=0$) we found the Weil-Petersson metric to be
\begin{equation}
g_{t\bar{t}}(0) = 0.19205 \pm 0.00104, \label{WPvalue}
\end{equation}
with $0.00104$ the standard error. Eq. \eqref{WPvalue} should be compared with the exact value ($0.1922\cdots$) that we obtained by computing the volume of the quintic 
as a function of $t$ using integrals on its 3-cycles (see Eq. \eqref{cande}). In Fig. \ref{batman1} we evaluated the Weil-Petersson metric on
the same region of the $t$-plane studied in Fig. \ref{candelas}. We used one Fubini-Study metric in the IS method for most smooth quintics on the $t$-plane except  
for those quintics that were close to developing double-point singularities. In this case, we used the refinement of the IS method with 19 Fubini-Study metrics.

\begin{figure}
\begin{center}
    \includegraphics[width=3in]{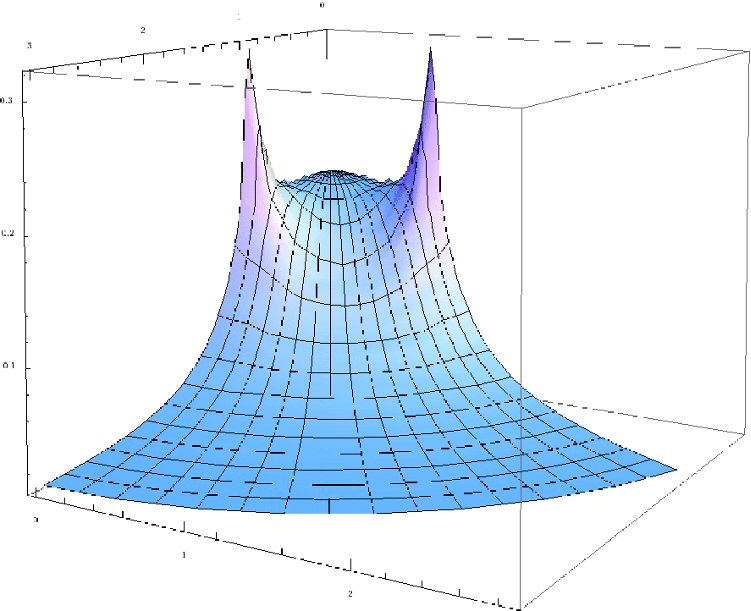}
\end{center}
  \caption{Evaluation of the Weil-Petersson metric (vertical axis) on the one-dimensional moduli space (horizontal $t$-plane) of Calabi-Yau quintic three-folds,  
$Z_{0}^5+Z_{1}^5+Z_{2}^5+Z_{3}^5+Z_{4}^5 - 5t Z_{0} Z_{1} Z_{2} Z_{3} Z_{4}$, using local deformations of the holomorphic form and the Importance Sampling method.}
    \label{batman1}
\end{figure}

In order to gain some insight on the computational complexity of this method, we compared these results with those of a simpler method.
In particular, another way to numerically compute the Weil-Petersson metric consists of approximating the K\"ahler potential on a local
neighborhood of the moduli space around the manifold where we are evaluating the WP metric.
To that end we evaluated the logarithm of the numerical volumes of several three-folds near the Fermat quintic ($t=0$), 
then we fitted a quadratic function to the random K\"ahler potential values, and computed its Hessian. 
Although this method is much simpler to implement, it is highly inefficient. For instance, if we evaluate the function
\begin{equation}
F(t,\overline{t})=-\log \left( \int_{Q_t} \nu_t\wedge \overline{\nu_t} \right),\label{log_vol}
\end{equation}
on 300 random quintics near $t=0$ and using 100,000 points on each three-fold $Q_t$ for the IS method,  
we find that the Hessian at $t=0$ is
$$
g_{t\bar{t}}(0) = 0.209693 \pm 0.03.
$$
This implies that using 15 times more points than in Eq. \eqref{WPvalue}, we were able to evaluate $g_{t\bar{t}}(0)$ with an error 30 times bigger. 
In Fig. \ref{quad_fit} we show the graph of the fitted function in Eq. \eqref{log_vol} using 300 points on the $t$-plane.

\begin{figure}
\begin{center}
    \includegraphics[width=3in]{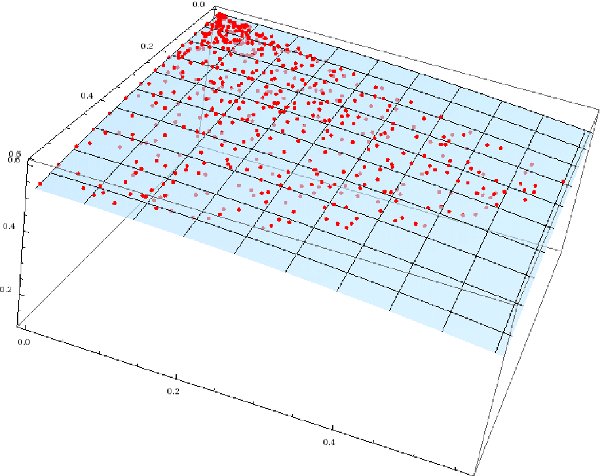}
\end{center}
  \caption{Quadratic approximation of the function $-\log(\int_X \nu_t\wedge\overline{\nu_t})$ around $t=0$, for 300 random Calabi-Yau Quintic 3-folds on the $t$ plane near $t=0$.}
    \label{quad_fit}
\end{figure}

\section{Numerical evaluation of the Weil-Petersson metrics using Donaldson's quantization approach}

A very powerful approach to constructing moduli spaces of polarized manifolds and the K\"ahler Weil-Petersson metrics associated with them
can be formulated within the mathematical framework of symplectic quotients and Geometric Invariant Theory (GIT).
The main strength of this approach is that it is not based on Eq. \eqref{kahlerpotential}, and therefore, it can be used to approximate
WP metrics on more general moduli spaces, such as spaces of constant scalar curvature K\"ahler (cscK) metrics and spaces of Hermite-Einstein metrics on vector bundles.
Although there exist several applications of this approach, we restrict our analysis to computing the WP metric on the moduli space of polarized Calabi-Yau manifolds.
This particular case helps to outline the main features of the method
and motivates further generalizations to other moduli spaces. 
%the \emph{infinite dimensional symplectic reduction} involves a moment map and a symplectic form on the quotient. In the case
%of K\"ahler manifolds, these objects are associated with a particular non-linear PDE and a K\"ahler Weil-Petersson metric defined on the symplectic quotient. 
%The sequence of \emph{finite dimensional symplectic quotients} involves analogous moment maps and compatible K\"ahler metrics on the symplectic quotient.
%In the classical limit of this quantized system, the moment map and K\"ahler metric converge to the corresponding objects in the
%infinite dimensional symplectic quotient (i.e. the aforementioned non-linear PDE and the K\"ahler Weil-Petersson metric).
%These finite dimensional symplectic quotients are particularly useful since the computational point of view because the associated moment maps and K\"ahler metrics
%can be computed exactly. Here, we show how one can use this constructible sequence of K\"ahler metrics to approximate the WP metric.

The core of the method consists of the correspondence between infinite dimensional
symplectic quotients and sequences of finite dimensional symplectic quotients built from appropriate quantizations of the classical systems.
In particular, the moduli space of polarized varieties can be constructed as an infinite dimensional symplectic quotient \cite{DonI}.
One can define this symplectic reduction by considering the infinite dimensional space of complex structures $\mathcal{J} = \{ J\colon TX\to TX\,\colon J^2=-1  \}$ 
that make $X$, endowed with $\omega$, a K\"ahler manifold. One also needs a certain complexification of the group of Hamiltonian diffeomorphisms on $(X,\,\omega)$,
which we denote here as $\mathrm{Ham}_{c}(X,\,\omega)$. 
If for simplicity one assumes that $H^1(X,\mathbb{R})=0$, then $\mathrm{Ham}_{c}(X,\,\omega)$ consists of the set of diffeomorphisms on $X$ 
that preserve the compatibility of the symplectic form $\omega$ with the complex structure $J$:
$$
\{ f\colon X\longrightarrow X\,\, \colon\, \exists Q\in C^{\infty}(X,\mathbb{R} )\,\, \mathrm{such\,\, that}\,\, f^\ast \omega = \omega +2i\partial\overline{\partial}Q \}.
$$
Therefore, the quotient of $\mathcal{J}$ by $\mathrm{Ham}_{c}(X,\,\omega)$ is the set of isomorphism classes of integrable complex structures on $(X,\omega)$.

The K\"ahler structure on $X$ induces one on $\mathcal{J}$ by integration. Now, given that $\mathcal{J}$ is a K\"ahler manifold and that the action of $\mathrm{Ham}_{c}(X,\,\omega)$
on $\mathcal{J}$ preserves its associated symplectic form, one can ask for the corresponding moment map. Fujiki and Donaldson showed that
$$
\mathrm{Moment\,\, map} = scal(\omega)-scal_0,
$$
where $scal(\omega)$ is the scalar curvature and $scal_0$ its average over $X$. In the case of K\"ahler manifolds $X$ with trivial canonical bundle, 
the zeros of the moment map correspond to the space of K\"ahler Ricci flat metrics on $X$. This means that the ``quotient'' $\mathcal{J}//\mathrm{Ham}_{c}(X,\,\omega)$ is also
the moduli space of Ricci flat metrics on the polarized Calabi-Yau manifold $(X,\,\omega)$.
Additionally, the symplectic form on this quotient gives rise to a K\"ahler metric on the moduli space which is identical to the Weil-Petersson metric. 

\subsection{Quantization of the infinite dimensional GIT problem}\label{CY-quant}
Donaldson showed in \cite{DonII} how the infinite dimensional quotient described above can be understood 
as the classical limit of a certain finite dimensional symplectic quotient. 
This finite dimensional construction can be defined using the set of pairs $(X,\cL_X)$,
with $X$ a projective Calabi-Yau manifold and $\cL_X$ a polarization with specified Euler-Poincar\'e characteristic $\chi(X,\cL_X^k)$.
In the limit of large $k$, $\chi(X,\cL_X^k)$ is fixed to be
$$
\chi(X,\cL_X^k)=\chi(k),
$$ 
where $\chi(k)\in \mathbb{Q}[k]$ is a specified polynomial of degree $n$.
We denote the set of pairs $(X,\cL_X)$ as $\mathcal{H}_\chi$.
The symplectic manifold that we use to construct the symplectic quotient consists of the subset of smooth projective manifolds in $\mathcal{H}_\chi$.
We denote such a space as $\mathcal{H}_\chi^\infty$, and we endowed it with a 
K\"ahler structure using the pullback of the Fubini-Study two-form coming from the corresponding projective embeddings of $X$.
In particular, for each element $(X,\cL_X)$ in $ \mathcal{H}_\chi$, and for $k$ large enough, one obtains
a set of non-unique embeddings into the projective space 
$$\imath \colon X \hookrightarrow \mathbb{P}H^0(X,\cL_X^k)^*=\mathbb{P}^N.$$
Using these embeddings one can construct explicitly the corresponding K\"ahler metric in $\mathcal{H}_\chi^\infty$ as follows.

Let $Hilb(N,\chi)$  be the Hilbert scheme of subschemes of  $\mathbb{P}^N$ with fixed Hilbert polynomial $\chi$.
By Grothendieck \cite{Vie}, we know that $Hilb(N,\chi)$ is a quasi-projective scheme that contains $\mathcal{H}_\chi$. 
Moreover, there exists a universal family
$Univ_{N,\chi}$ over $Hilb(N,\chi)$, such that $Univ_{N,\chi} \subset Hilb(N,\chi)\times \mathbb{P}^N$. Also, there exists a natural map from 
$Univ_{N,\chi}$ to $\mathcal{H}_\chi$ and another one from $Univ_{N,\chi}$ to $\mathbb{P}^N$:
\begin{equation}
\begin{array}{lll}
Univ_{N,\chi} & \overset{\pi_2}{\longrightarrow } &  \mathbb{P}^N \\
\downarrow \pi_1 &  & \\
\mathcal{H}_\chi & & \label{diagram_l}.
\end{array}
\end{equation}
Using these maps and fixing the value of $k$, one can define the compatible K\"ahler metric on $\mathcal{H}_\chi^\infty$ as
\begin{equation} \label{pull-back}
\Omega_k= \pi_{1*}( \pi_2^*\omega_{FS} \cap [Univ_{N,\chi}]). 
\end{equation}

Similarly to the infinite dimensional construction, the action of $SL(N+1,\mathbb{C})$ on $\mathbb{P}^N$ induces a Hamiltonian action on
$\mathcal{H}_\chi^\infty$ which preserves the symplectic form defined in Eq. \eqref{pull-back}.
This group action gives rise to a moment map $$\mu \colon  \mathcal{H}_\chi^\infty\rightarrow Lie(SU(N+1))^*.$$ 
In particular, given an orthonormal basis of sections $\{s_\alpha\}$ of $H^0(X,\mathcal{L}^k)$, one can write this map as 
\begin{equation}\label{moment}
\mu(X)= \delta_{\alpha \bar \beta} - \frac{N+1}{\rm Vol(X)}\int_X \frac{s_\alpha \bar{s}_{\bar{\beta}}} {\sum_\gamma |s_\gamma|^2} \frac{\omega_{FS}^n}{n!}.
\end{equation}
Not every manifold $X$ admits a solution of this moment map. Only those
projective manifolds $(X,\mathcal{L}^k)\in \mathcal{H}_\chi^\infty$ that are polystable in the GIT sense admit a solution of Eq. \eqref{moment}. 
We denote the set of these polystable manifolds as $Hilb(N,\chi)^{ps}$, and we name a particular solution
of Eq. \eqref{moment} as a \emph{balanced embedding}.

There exists an extensive literature regarding the properties of these solutions, e.g. see \cite{YauI,DonII,DonV,Sa}. 
In particular, the balanced projective embedding defined in Eq. \eqref{moment} induces by restriction a metric on $X$ which is well studied.
%For instance, we know that
%for any smooth manifold $X$, with discrete ${\rm Aut}(X,\cL)$ and with a K\"ahler Ricci flat metric in the class $c_1(\mathcal{L})$,  
%the corresponding sequence of balanced metrics indexed by $k$ converges to the Calabi-Yau metric when $k\rightarrow \infty$. 
We know that
for smooth Calabi-Yau varieties $X$ with a K\"ahler Ricci flat metric in the class $c_1(\mathcal{L})$,  
the corresponding sequence of balanced metrics indexed by $k$ converges to the Ricci flat metric
in the limit $k\rightarrow \infty $ \footnote{This statement is true
only if ${\rm Aut}(X,\cL)$ is discrete. However, in the case of a similar class of balanced metrics, known as $\nu$-balanced metrics and 
defined in the next subsection, the statement is true for any ${\rm Aut}(X,\cL)$; see \cite{Ke}.}. 
Moreover, this sequence of metrics is computable. In order to show how to compute the elements of this sequence, we first need to introduce a few definitions.
Let ${\rm Met}(\mathcal{L}^k)$ be the space of hermitian metrics on $\mathcal{L}^k$, and let ${\rm Met}(H^0(X,\mathcal{L}^k))$ be the space 
of metrics on the vector space $H^0(X,\mathcal{L}^k)$. There are two natural maps between ${\rm Met}(\mathcal{L}^k)$ and ${\rm Met}(H^0(X,\mathcal{L}^k))$.
First, the Hilbertian map $Hilb_k$ is defined as 
$$
Hilb_k\colon  {\rm Met}(\mathcal{L}^k)  \rightarrow {\rm Met}(H^0(X,\mathcal{L}^k)),
$$
with $Hilb_k(h)(s,\bar s) = \int_X |s|^2_h \frac{c_1(h)^n}{n!}$ and $c_1(h)$ the curvature of the metric $h \in {\rm Met}(\mathcal{L}^k) $. 
Second, the  injective Fubini-Study map $FS_k$ is defined as 
$$
FS_k\colon {\rm Met}(H^0(X,\mathcal{L}^k)) \rightarrow  {\rm Met}(\mathcal{L}^k) 
$$ 
with $FS_k(H)=\frac{1}{k}\log \left( \sum_{i=1}^{N_k} |S_i^H|^2 \right)$, and $S_i^H$ is an $H$-orthonormal basis\footnote{This definition is  independent of the choice of the basis.} of holomorphic sections of $H^0(X,\mathcal{L}^k)$.
The existence of a solution of Eq. \eqref{moment} is equivalent to the existence of a Hermitian metric $H_k$ on
$H^0(X,\mathcal{L}^k)$ that satisfies $Hilb_k(FS_k(H_k))=H_k$. Similarly, a solution of Eq. \eqref{moment} is equivalent to a fixed point of the $T$-map, 
with $T=Hilb_k \circ FS_k$. 
We know by the work of \cite{DonII,DonV,Sa} that the $T$-map admits a unique fixed point, the balanced metric $H_k$. 
Furthermore, the sequence of metrics $T^r(H^{(0)}_k)$, 
defined by the $r^{th}$ iteration of the $T$-map for a generic initial metric $H^{(0)}_k\in {\rm Met}(H^0(X,\mathcal{L}^k))$, converges exponentially fast in $r$ to the fixed
point $H_k$.
Donaldson has shown how these results can be translated into numerical methods that approximate K\"ahler Ricci flat metrics on $X$;
see \cite{Don}.

The moduli space of polarized Calabi-Yau manifolds that arises as the quotient 
$$Hilb(N,\chi)^{ps}//SL(N+1)
$$ 
is identical to the moduli space that arises as the infinite-dimensional symplectic ``quotient'' 
$$
\mathcal{J}//\mathrm{Ham}_{c}(X,\,\omega);
$$ 
additionally, the zeroes of both  moment maps describe the same K\"ahler Ricci flat metric on $X$ in the $k\to\infty$ limit.
Therefore, it is natural to assume that the sequence of K\"ahler metrics 
that $Hilb(N,\chi)^{ps}//SL(N+1)$ inherits from its symplectic reduction will converge 
to the K\"ahler metric that $\mathcal{J}//\mathrm{Ham}_{c}(X,\,\omega)$ inherits from its corresponding symplectic reduction.
As the metric that $\mathcal{J}//\mathrm{Ham}_{c}(X,\,\omega)$ inherits is the Weil-Petersson metric itself, 
we expect that the metric inherited from $(\mathcal{H}_\chi^\infty,\, \Omega_k)$ will approximate the WP metric
in the limit of large $k$.

Here, we do not provide a formal proof that this sequence of K\"ahler metrics actually converges to the WP metric. 
Nevertheless we construct such metrics, numerically compute some examples
and compare them with the exact WP metric. This constitutes a heuristic approach to understanding how both types of metrics relate to each other,
which we expect will contribute to the development of a formal proof of convergence.

\subsection{Quantization of the Weil-Petersson metric on the moduli space of Calabi-Yau manifolds}\label{mu-bal1}

In this subsection we derive an explicit 
expression for the K\"ahler metric that $Hilb(N,\chi)^{ps}//SL(N+1)$
inherits from the symplectic reduction of $\mathcal{H}_\chi^\infty$ 
endowed with the K\"ahler metric Eq. \eqref{pull-back}. We denote this \emph{quantized Weil-Petersson} metric as $\Omega_{k,\cT}$.
In particular, we provide an analytical formula for the restriction of $\Omega_{k,\cT}$ to the tangent space of a point in the moduli space
of polarized Calabi-Yau manifolds. 

Let $T_t\cT$ be the space of infinitesimal deformations of the complex structure on $(X_t,\, [\omega])$, i.e. 
the tangent space of a family of Calabi-Yau manifolds at $X_t$. Let 
$L^k$ be the complex line bundle on the differential manifold $X$ that has $c_1(L^k)=c_1(\cL^k)$.
Then, the holomorphic line bundle $\cL^k$ on $X_t$ and the complex structure $t$ on $X$ yield  
a Dolbeault operator $\bar{\partial}_t\colon \Omega^{p,q}(\cL^k)\to \Omega^{p,q+1}(\cL^k)$. 
In other words, $\cL^k$ can be defined as a pair $(L^k,\, \bar{\partial}_t)$ on $X$.  
Any deformation $v\in T_t\cT$ of the complex structure on $X_t$ induces a deformation of the Dolbeault operator 
$$
\bar{\partial}_{t+v} =\bar{\partial}_t  + v\partial_t +O(v^2).
$$ 
Therefore, if $\{ s_\alpha \}_{\alpha=0}^N$ spans $\ker \bar{\partial}_t = H^{0}(X_t,\, \cL^k)$, and $\{ s_\alpha + \delta_v s_\alpha \}_{\alpha=0}^N$
spans $\ker \bar{\partial}_{t+v} = H^{0}(X_{t+v},\, \cL^k)$, we can compute the infinitesimal deformations of the sections $\delta_v s_\alpha$ as
$$
\delta_v s_\alpha =-\overline{\partial}_{t}^{-1} \left(v \partial s_\alpha \right) +O(v^2).
$$ 
As the hermitian metric $h$ on $L^k$ induces a $L^2$ metric on $\Omega^0(X,L^k)$ and an $L^2$-completion $L^2(X,L^k)$ of the space of sections,
one can define precisely the inverse operator $\overline{\partial}_{t}^{-1}$ on the orthogonal complement $\ker\overline{\partial}_t^\perp \subset L^2(X, L^k)$.
Now, we can introduce a local formula for $\Omega_{k,\cT}$.

If $H_t\in {\rm Met}(H^0(X_t,\mathcal{L}^k))$ is the balanced metric of order $k$, and $v_1,v_2 \in T_t\cT$ are infinitesimal deformations in $H^{0,1}_{\bar{\partial}}(TX_{t})$, 
the inner product associated with the quantized WP K\"ahler metric $\Omega_{k,\cT}$ is
\begin{equation}
\Omega_{k,\cT}(v_1,v_2)=k^n H^{\alpha,\beta}_t \int_{X_t}  \langle \bar{\partial}^{-1} (v_1 \partial s_\alpha),{\bar{\partial}}^{-1} (v_2\partial s_\beta ) \rangle_{FS_k(H_t)} \frac{(\frac{1}{k}c_1(FS_k(H_t))^n}{n!} \label{def4},
\end{equation}
for $\{ s_\alpha \}$ an orthonormal basis with respect to $H_t$.

In practice, it is more convenient to use a slightly different metric that employs the 
volume form $\nu_t\wedge\overline{\nu_t}$ instead of the Fubini-Study volume form. 
Although both metrics converge to the same asymptotic in the large $k$ limit,
this second metric has several advantages in numerical applications.
In order to define this metric, one first needs a modified notion of the balanced metric, the $\nu$-\emph{balanced metric}.
In particular, if the $\nu$-Hilbertian map is
\begin{equation}
Hilb_{k,\nu}(h) = \int_X \vert s\vert^2_h \nu_t\wedge\overline{\nu_t}, \label{tamp}
\end{equation}
the $\nu$-balanced metric of order $k$ is a fixed point $H$ of the map $T_\nu=Hilb_{k,\nu}\circ FS_k$; see \cite{Ke}.
Then, in an abuse of notation, we define the quantized Weil-Petersson metric of order $k$ (modification of Eq. \eqref{def4}) by the inner product
\begin{equation}
\Omega_{k,\cT}(v_1,v_2)=k^n H^{\alpha,\beta}_t \int_{X_t}  \langle \bar{\partial}^{-1} (v_1 \partial s_\alpha),{\bar{\partial}}^{-1} (v_2\partial s_\beta ) \rangle_{FS_k(H)} \nu_t\wedge\overline{\nu_t} \label{def6},
\end{equation}
where $H$ denotes here the $\nu$-balanced metric.

As the $\nu$-balanced metric of order $k=\infty$, that we denote as $\omega_{k=\infty,t}$, is the Calabi-Yau metric
one can infer from the asymptotic expansion of the Bergman kernel \cite[Remark 5.1.5]{M-M} that  
$$
\Vert \omega_{FS(H)}- \omega_{\infty,t}\Vert_{C^\infty} =O\left(\frac{1}{k^2}\right).
$$
Therefore, it follows that the sequence of modified metrics introduced in Eq. \eqref{def6} and the sequence defined in Eq. \eqref{def4} converge to the same metric.
In the following subsections, we introduce numerical methods to evaluate explicitly the metric defined in Eq.  \eqref{def6} and compute several examples
on the one-dimensional family of quintics.

\subsection{Numerical computation of the quantized Weil-Petersson metric}

In \cite{Don}, Donaldson showed how to numerically compute \emph{balanced metrics} and $\nu$-balanced metrics, in order to approximate cscK metrics on varieties and K\"ahler Ricci flat metrics on Calabi-Yau manifolds. Such metrics can be constructed explicitly if one has analytic control on the projective embeddings and can evaluate integrals on $X$. 
Similar technical difficulties arise if we want to evaluate quantized Weil-Petersson metrics on moduli spaces (i.e. Eq. \eqref{def6}). For any basis of sections $\{ s_\alpha \}_{\alpha=1}^{N+1}$, and definite positive Hermitian matrix $H_{\alpha\bar{\beta}}$ on the vector space $H^0 (X_{t_0},\,\cL^k)$, one can introduce a $L^2$-product on $\Omega^0(\cL^k)$ defined as
\begin{equation}
\langle \sigma_1,\,\sigma_2\rangle = \frac{1}{{\rm Vol}(X_{t_0})}\int_X \frac{\sigma_1\overline{\sigma_2}}{(H^{-1})^{\bar{\gamma}\delta} s_\delta\bar{s}_{\bar{\gamma}}} \nu_{t_0}\wedge\overline{\nu}_{t_0}, \label{L2bal}
\end{equation}
with ${\rm Vol}(X_{t_0})=\int_X \nu_{t_0}\wedge\overline{\nu}_{t_0}$. 
%The $L^2$-completion of the space of smooth sections $C^{\infty}(X_{t_0}, \cL^k)$ is denoted as $L^2(X_{t_0}, \cL^k)$.
If the restriction of Eq. \eqref{L2bal} on $H^0(X_{t_0},\,\cL^k) \subset L^2(X_{t_0}, \cL^k)$ is identical to the inner product defined by $H$, one says that the Fubini-Study metric defined by $H$ on $\mathbb{P}H^0(X_{t_0},\,\cL^k)^\ast$ is $\nu$-\emph{balanced}\footnote{In this section, we refer to the $\nu$-balanced metrics simply as ``balanced metrics.''} as we have seen in Section \ref{mu-bal1}. One can find the balanced metric on $X_{t_0}\hookrightarrow\mathbb{P}^N$ by introducing any initial definite positive Hermitian matrix $H(0)$ on $H^0 (X_{t_0},\,\cL^k)$, and iterating the map, $T\colon {\rm Met} (H^0 (\cL^k))\to {\rm Met} (H^0 (\cL^k))$
\begin{equation}
H(q+1)_{\alpha\bar{\beta}}=T(H(q))_{\alpha\bar{\beta}} = \frac{N+1}{\mathrm{Vol}(X_{t_0})}\int_X \frac{s_\alpha\bar{s}_{\bar{\beta}}}{(H(q)^{-1})^{\bar{\gamma}\delta} s_\delta\bar{s}_{\bar{\gamma}}} \nu_{t_0}\wedge\overline{\nu}_{t_0},\label{tmap}
\end{equation}
up to reaching convergence with the limit point $H_{t_0}=H({+\infty})$. \\
Let $\{ t_a \}_{a=1}^{a=\dim \cT}$ be a local coordinate system on the moduli space $\cT$ and $v=v^a\partial_{t_a} \in T_{t_0}\mathcal{T}$ be an infinitesimal deformation
of $t_0$.  In order to find the infinitesimal deformation of the balanced embedding of $X_{t_0 + v}$ into $\mathbb{P}^N$ and evaluate 
Eq. \eqref{def6}, it is convenient to work with a family of line bundles on $X$. If $\pi\colon \cX\to \cT$ denotes a family of complex structures on $X$, with $\pi^{-1}(t)=X_t$,
there exists a holomorphic line bundle $\cS^k \to\cX$ such that $\cS^k \vert_t = \cL_t^k \to X_t$. In other words, the restriction of $\cS^k$
to the fibers of $\pi\colon \cX\to \cT$ is identical to the holomorphic polarization $\cL_t^k$ on $X_t$. The natural Hermitian structure on $\cL_t^k\to X_t$ whose curvature is the corresponding K\"ahler Ricci flat metric, lifts to a hermitian structure on $\cS^k\to \cX$. When we approximate the K\"ahler Ricci flat metrics on $X_t$ by balanced metrics, $\cS^k\to \cX$ also admits a compatible hermitian structure.
More precisely, if $\{ s_\alpha(t,\bar{t}) = \eta_\alpha(t,\bar{t})\hat{e}_t \}_{\alpha=1}^{N+1}$ is a basis of holomorphic sections for $H^0 (X_{t},\,\cL^k)$, $\hat{e}_t$ is the holomorphic frame in a local trivialization, the parameters $t$ denote the moduli
dependence, and $H_t$ is the associated balanced matrix, we can endow $\cS^k \to \cX$ with the hermitian metric
\begin{equation}
h_t=\frac{\hat{e}_t\otimes \hat{e}_t^\ast}{{(H_t^{-1})^{\bar{\gamma}\delta} s(t)_\delta\bar{s}(\bar{t})_{\bar{\gamma}}}}.\label{hermitianmetric}
\end{equation}
Therefore, given the diffeomorphism between $X_{t_0}$ and $X_{t_0+v}$, defined in local holomorphic coordinate charts (see Eq. \eqref{diffeo}), as
$$
y^i = w^i + v^{a}\vartheta_a^{i}(w,\bar{w}) + O(v^2),
$$
one can compute the infinitesimal deformation of the embedding $X_{t_0+v}\hookrightarrow \mathbb{P}^N$, as the 
covariant derivative of $s_\alpha(t)$ 
\begin{equation}
\nabla_v \eta_\alpha\hat{e}_t = v^a\frac{\partial \eta_\alpha}{\partial t_a}\hat{e}_t + v^a h^{-1}_t \frac{\partial h_t}{\partial t_a} \eta_\alpha\hat{e}_t, \label{cov_der}
\end{equation}
where $\frac{\partial \eta_\alpha}{\partial t_a} =\frac{\partial \eta_\alpha}{\partial w^i} \vartheta_a^{i}(w,\bar{w})$.
In other words, if $\hat{e}(y)$ is a holomorphic frame for $\cL_{t_0+v}^k\to X_{t_0+v}$, we can write the basis of holomorphic sections as
\begin{equation}
s_\alpha (y) = \eta_\alpha(y)\hat{e}(y) = \eta_\alpha (w) \hat{e}(w) + \nabla_v \eta_\alpha(w,\bar{w})\hat{e}(w) + O(v^2). 
\end{equation}
The proof is straightforward. Thus, $\nabla_v \eta_\alpha \hat{e}$ are smooth sections in $L^2(X_{t_0}, \cL^k)$ that represent components of vector fields along $T^{1,0}\vert_{X_{t_0}} \mathbb{P}^N$. The sections $\nabla_v \eta_\alpha \hat{e}\in L^2(X_{t_0}, \cL^k)$, can be expressed as the sum 
of a holomorphic section plus a non-holomorphic section, because of the decomposition $$L^2(X_{t_0}, \cL^k)=H^0(X_{t_0}, \cL^k) \oplus H^0(X_{t_0}, \cL^k)^{\perp}$$ under the $L^2$-metric defined in Eq. \eqref{L2bal}. As we need the normal components of $\nabla_v \eta_\alpha \hat{e}$
to $H^0(X_{t_0}, \cL^k)$, we have to project out the holomorphic part, $P_{t_0}\nabla_v \eta_\alpha \hat{e} \in H^0(X_{t_0}, \cL^k)$. The holomorphic part of $\nabla_v \eta_\alpha \hat{e}$ can be computed using the Bergman kernel projector $P_{t_0}\colon L^2(X_{t_0}, \cL^k)\to H^0(X_{t_0}, \cL^k)$. In an orthonormal basis $\{ s_\alpha^\prime \}_{\alpha=1}^{N+1}$ one can express $P_{t_0}$ as
\begin{equation}
P_{t_0}(\sigma) = \frac{N+1}{{\rm Vol}(X_{t_0})} \sum_\alpha s_{\alpha}^\prime \int_X \frac{\sigma\overline{s}^{\prime\alpha}} {(H_{t_0}^{-1})^{\bar\gamma \delta} s^\prime_\delta\overline{s}^{\prime}_{\gamma}} \nu_{t_0}\wedge\overline{\nu}_{t_0}. \label{proj}
\end{equation}
Therefore, the term $(Id-P_{t_0})\nabla_v \eta_\alpha \hat{e}$ denotes the projection of $\nabla_v \eta_\alpha \hat{e}$ onto the orthogonal complement in $\Gamma\left( T^{1,0}\vert_X \mathbb{P}^N \right)$ of the subspace defined by the infinitesimal action of $GL(N+1,\mathbb{C})$, which is isomorphic to $H^0(X_{t_0}, \cL^k)$. Hence, the quantized Weil-Petersson metric in Eq. \eqref{def6} can be written as:
\begin{equation}
\Omega_{k,\cT}(v_1,\, v_2) =\frac{\left( H_{t_0}^{-1}\right) ^{\bar{\beta}\alpha} }{{\rm Vol}(X_{t_0})} \int_X \frac{\left( (Id-P_{t_0})\nabla_{v_1} \eta \right)_\alpha \overline{\left( (Id-P_{t_0})\nabla_{v_2} \eta \right)_{\bar{\beta}}} }{(H_{t_0}^{-1})^{\bar{\gamma}\delta} \eta_\delta\bar{\eta}_{\bar{\gamma}}} \nu_{t_0}\wedge\overline{\nu}_{t_0}. \label{quantmet}
\end{equation}
Here, the basis of sections $\{ s_\alpha=\eta_\alpha \hat{e}  \}_{\alpha=1}^{N+1}$ is not necessarily orthonormal, although due to the simplicity of $P_{t_0}$ when expressed in an orthonormal basis (see Eq. \eqref{proj}), it is convenient to work with $\{ s_\alpha \}_{\alpha=1}^{N+1}$ orthonormal (where $H_{t_0}=Id$).

Thus, evaluating Eq. \eqref{quantmet} involves the following algorithm:
\begin{enumerate}
\item Building an explicit basis of sections $\{ s_\alpha=\eta_\alpha \hat{e}  \}_{\alpha=1}^{N+1}$ for $H^0 (X_{t_0},\,\cL^k)$, with $k=1,\ldots, k_{max}$, and $k_{max}$ some maximum value of $k$ that one can handle numerically.
\item Developing a numerical algorithm to evaluate integrals on $X$ under the measure $\nu_{t_0}\wedge\overline{\nu}_{t_0}$, as we did in Section 3.1.
\item Computing the balanced metric $H$ by iteration of the $T$ map \eqref{tmap}.
\item Choosing a basis of infinitesimal diffeomorphisms $\vartheta_a^{i}(w,\bar{w})$ on $X_{t_0}$, isomorphic to the basis $\frac{\partial}{\partial t_a}$ for $T_{t_0}\cT\simeq H^{1}(X_{t_0},\Omega^{n-1})$.
\item Solving the linearized balanced equations for $\partial_{t_a} H_{\alpha\bar{\beta}}$ at $t=t_0$:
\begin{eqnarray}\hspace*{1cm}
\frac{\partial H_{\alpha\bar{\beta}}}{\partial t_a} &=& c_0 \int_X \frac{\nabla_a \eta_\alpha \hat{e} \bar{s}_{\bar{\beta}}}{(H^{-1})^{\bar{\gamma}\delta} s_\delta\bar{s}_{\bar{\gamma}}} \nu_{t_0}\wedge\overline{\nu}_{t_0}
+c_0 \int_X \frac{s_\alpha\bar{s}_{\bar{\beta}}}{(H^{-1})^{\bar{\gamma}\delta} s_\delta\bar{s}_{\bar{\gamma}}} \frac{\partial \nu_{t}}{\partial t_a}\wedge\overline{\nu}_{t_0}  \nonumber \\
&&-\frac{c_0}{ \mathrm{Vol}(X_{t_0})} \int_X \frac{\partial \nu_{t}}{\partial t_a}\wedge\overline{\nu}_{t_0} \times \int_X \frac{s_\alpha\bar{s}_{\bar{\beta}}}{(H^{-1})^{\bar{\gamma}\delta} s_\delta\bar{s}_{\bar{\gamma}}} \nu_{t_0}\wedge\overline{\nu}_{t_0}, \label{linbal}
\end{eqnarray}
where $c_0=\frac{N+1}{\mathrm{Vol}(X_{t_0})}$. We obtain these relationships by differentiating the condition on $H$ to be a fixed point of the $T$ map and the fact that the Bergman function associated with the balanced metric is constant. Eq. \eqref{linbal} is a non-trivial system of linear equations, as $\partial_{t_a} H_{\alpha\bar{\beta}}$ is contained in the $\nabla_a \eta_\alpha$ term. One can solve Eq. \eqref{linbal} by using Gauss' elimination method, or one can solve it iteratively by setting  $\partial_{t_a} H_{\alpha\bar{\beta}}(q=0)=0$ as the initial value and interpreting Eq. \eqref{linbal} as a linearized $T$ map.
\item Computing $\nabla_a \eta_\alpha$, given $\frac{\partial H_{\alpha\bar{\beta}}}{\partial t_a}$, and its projection $(Id-P_{t_0})\nabla_{a}\eta_\alpha$  using Eq. \eqref{proj}.
\item And finally, evaluating the inner products Eq. \eqref{quantmet}.
\end{enumerate}

\subsection{Example: the family of quintics}\label{example2}

Here, we evaluated the algorithm in a similar fashion as we did in Section \ref{method1}.
In particular, we implemented the algorithm that we have just described, for the family of quintic three-folds $Q_t$ in $\mathbb{P}^4$ 
defined by the polynomial
$$
P(Z)=Z_{0}^5+Z_{1}^5+Z_{2}^5+Z_{3}^5+Z_{4}^5 - 5t Z_{0} Z_{1} Z_{2} Z_{3} Z_{4}.
$$
We studied the region of the $t$-plane given by $0< \vert t \vert \leq 3 $ and $0\leq \arg(t) <  2\pi /5$, as we 
did in the examples of sections 2 and 3. We divided the region in a lattice of more than 300 points, and 
computed the corresponding balanced metrics for embeddings in linear spaces of sections, up to degree $k=6$.
We chose monomials of degree $k$ defined on $\mathbb{P}^4$, modulo the ideal generated by $P(Z)$, as the 
basis of sections. We evaluated the integrals that appear in the $T$ map in Eq. \eqref{tmap} by using the Monte Carlo method 
described in Section 3. In order to compute the variation of the sections $\frac{\partial s_\alpha}{\partial t}$, 
we used the infinitesimal diffeomorphism defined by Eq. \eqref{vectorfield}, with
\begin{equation}
\frac{\partial \eta_\alpha}{\partial t} = \sum_{i=1}^4 \frac{\partial \eta_\alpha}{\partial w_i} \frac{\partial w_i}{\partial t}=
\sum_{i=1}^4 \frac{\partial \eta_\alpha}{\partial w_i} \vartheta^{i}(w,\bar{w}). \label{pertpol}
\end{equation}
For this family of quintics, Proposition \ref{prop1} defines the vector field $\vartheta^{i}(w,\bar{w})$ as
\begin{equation}
\vartheta^{i}(w,\bar{w}) = - \frac{G^{i\bar{\jmath}} \frac{\partial \bar{p}(\bar{w})}{\overline{\partial} \overline{w}_{\bar{\jmath}}}}{G^{m\bar{n}} \frac{\partial \bar{p}(\bar{w})}{\overline{\partial} \overline{w}_{\bar{n}}}\frac{\partial p(w)}{\partial w_m }
} (-5 w_{1} w_{2} w_{3} w_{4}), \label{vectorfieldquintic}
\end{equation}
with $G^{i\bar{\jmath}}$ the inverse of the Fubini-Study metric in $\mathbb{P}^4$, and $w_i=Z_i/Z_0$ local coordinates on $Q_t\subset\mathbb{P}^4$.
Given the Hermitian metric $h_t$ defined in Eq. \eqref{hermitianmetric}
%$$
%h_t=\frac{\hat{e}_t\otimes \hat{e}_t^\ast}{{(H^{-1})^{\bar{\gamma}\delta}(t,\bar{t}) s(t)_\delta\bar{s}(\bar{t})_{\bar{\gamma}}}},
%$$
we can compute the covariant derivative $\nabla_t \eta_\alpha \hat{e}$ as
\begin{eqnarray*}
\nabla_t \eta_\alpha \hat{e} &=& \sum_{i=1}^4 \frac{\partial \eta_\alpha}{\partial w_i} \vartheta^{i}(w,\bar{w})\hat{e}\\
&& - \frac{\partial}{\partial t} \left({(H_t^{-1})^{\bar{\gamma}\delta} \eta(t)_\delta\bar{\eta}(\bar{t})_{\bar{\gamma}}} \right) \frac{\eta_\alpha \hat{e}}{{(H_t^{-1})^{\bar{\gamma}\delta} \eta(t)_\delta\bar{\eta}(\bar{t})_{\bar{\gamma}}}},
\end{eqnarray*}
which we evaluated by using Eq. \eqref{pertpol}, and solving the linearized balanced equations in Eq. \eqref{linbal} for $\frac{\partial H_{\alpha\bar{\beta}}}{\partial t}$. The method that we implemented to solve the linearized balanced equations consisted of iterating the linearized $T$ map; this iterating scheme reached good estimates of the solutions within 5 or 6 iterations.
\begin{figure}
\begin{center}
\begin{tabular}{ccc}
    \includegraphics[width=1.5in]{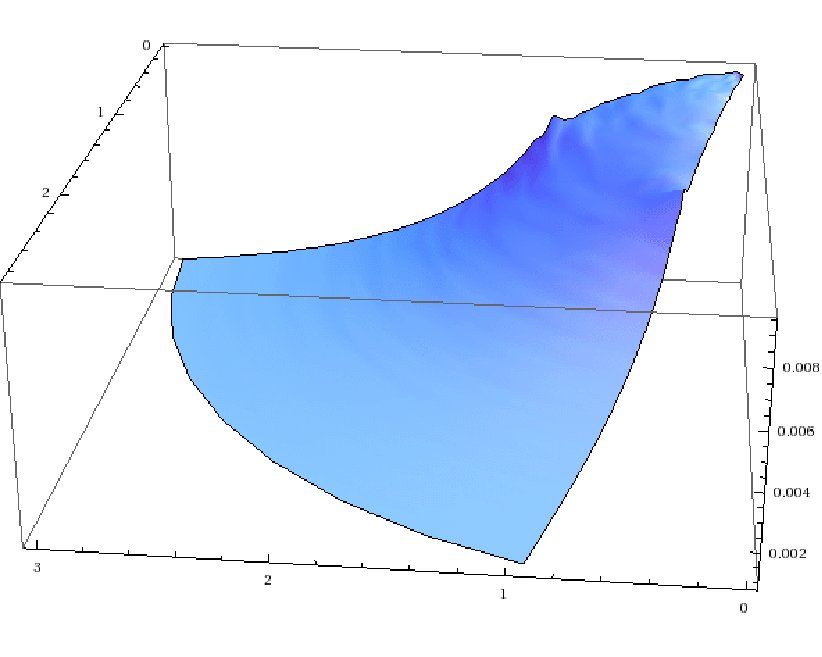}&
    \includegraphics[width=1.5in]{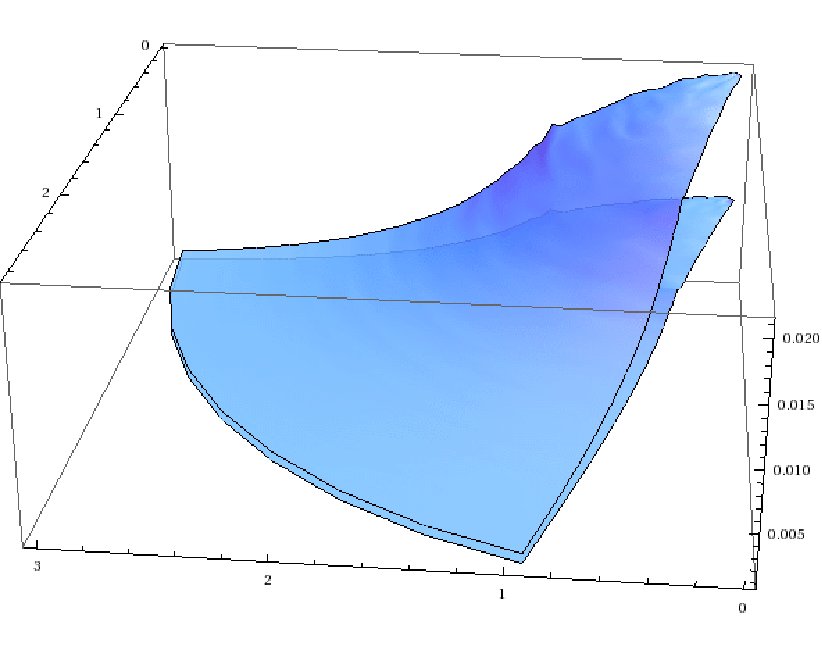}&
    \includegraphics[width=1.5in]{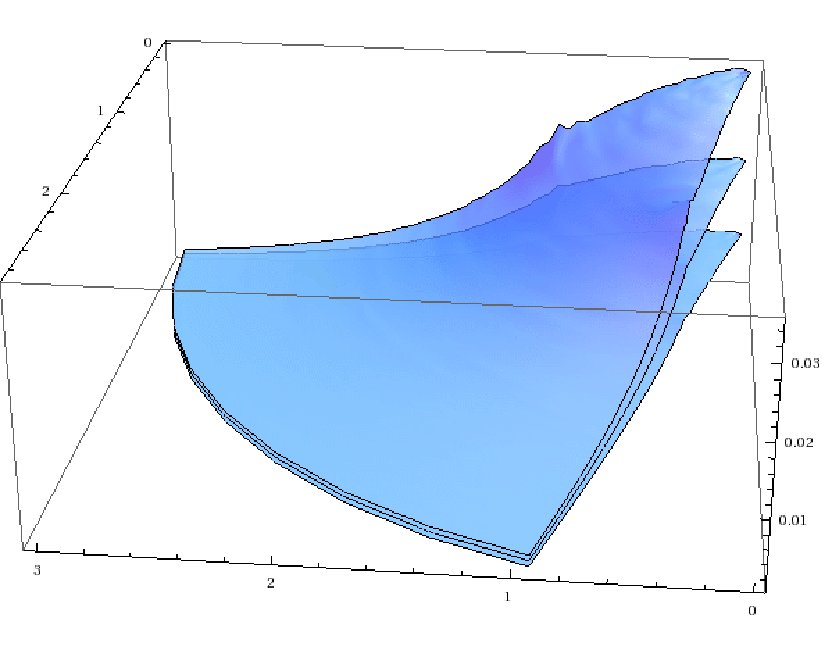}\\ \\
    \includegraphics[width=1.5in]{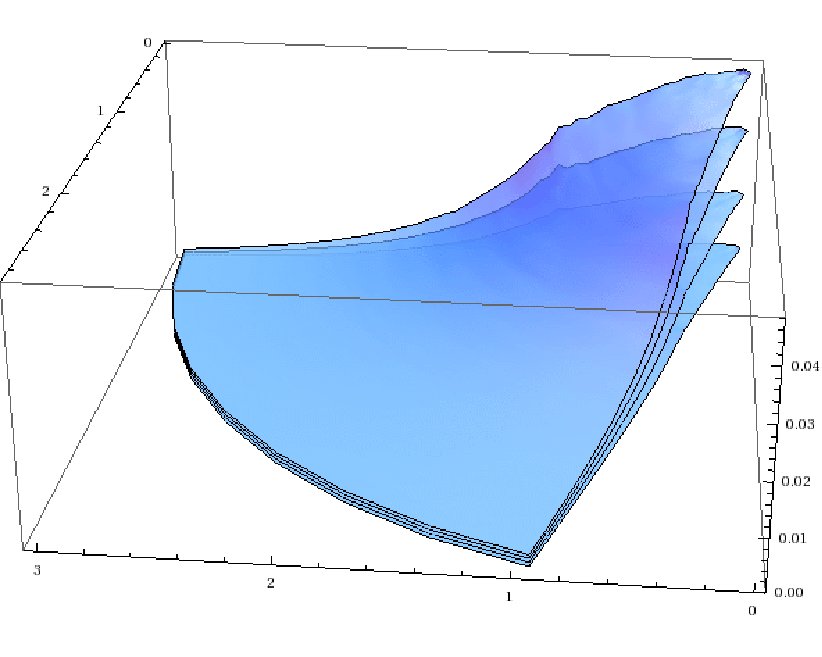}&
    \includegraphics[width=1.5in]{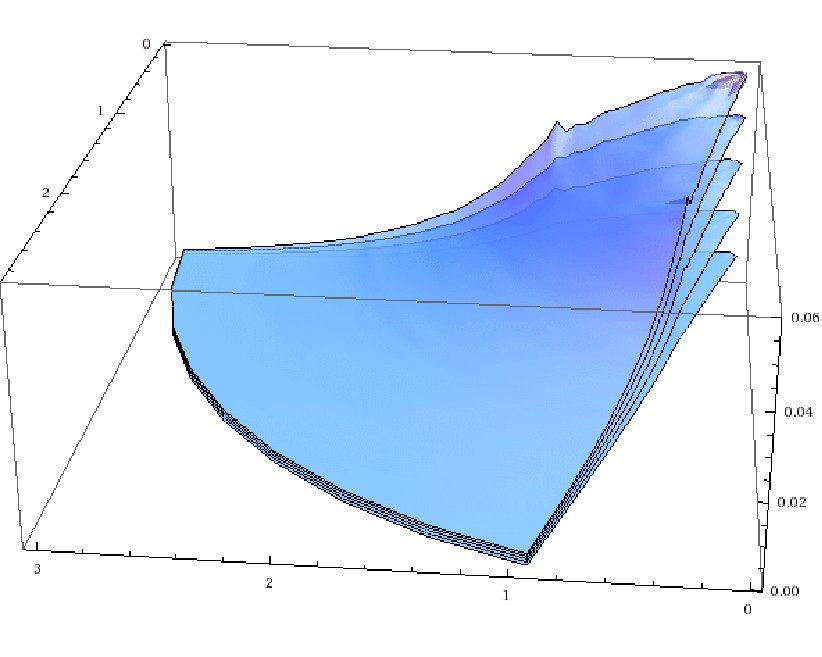}&
    \includegraphics[width=1.5in]{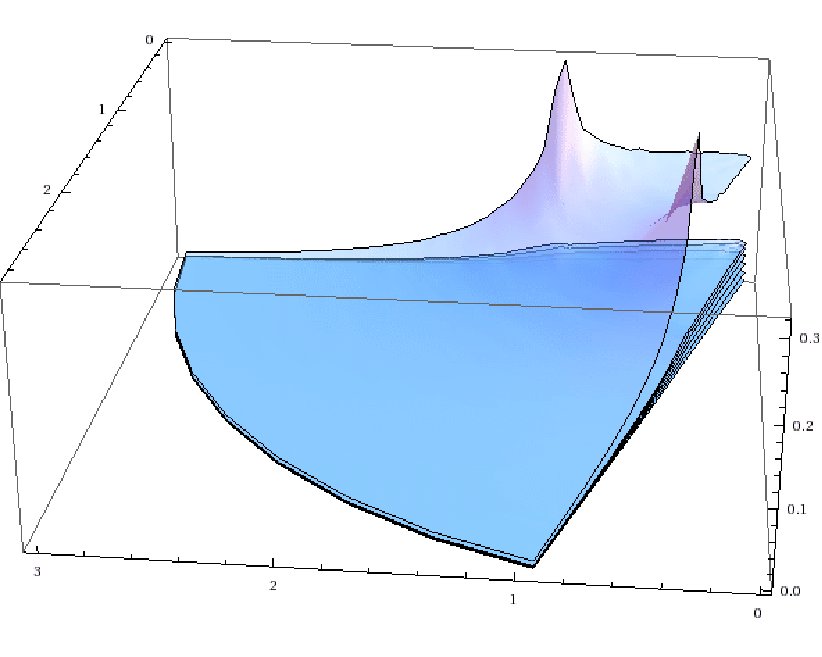}
\end{tabular}
\end{center}
  \caption{Quantized Weil-Petersson metrics (vertical axis) on the $t$-plane (horizonal plane), 
   of Calabi-Yau Quintic threefolds  $Z_{0}^5+Z_{1}^5+Z_{2}^5+Z_{3}^5+Z_{4}^5 - 5t Z_{0} Z_{1} Z_{2} Z_{3} Z_{4}$,
  for  $k=1$, 2, 3, 4, 5, and 6.}
    \label{quantized_metrics}
\end{figure}
In Fig. \ref{quantized_metrics} we plot the sequence of metrics $\Omega_{k,\cT}$, defined in Eq. \eqref{quantmet}, for $k=1,2,\ldots, 6$. The time taken
to compute each value $\Omega_{k,\cT}(t,\bar{t})$ per point in the $t$-plane was approximately equal to $4$ times the time needed to compute the balanced metric. One can observe that for large $\vert t \vert$, the rate of convergence of the sequence is higher than in other
regions of the $t$-plane. In points near the Fermat quintic, $t=0$, and for $k=6$, the quantized K\"ahler metric is approximately $0.07$ vs the exact value $0.19$. 
If we extrapolate the rate of convergence observed for the first six values of $k$, one expects to attain deviations smaller than $0.01$ in this region of the $t$-plane when $k>12$. The worst rate of convergence is located near the points $t=\exp(2\sqrt{-1}\pi\mathbb{Z}/5)$, where the quintic develops double point singularities. In this region of the family, the approximation of the corresponding Ricci flat metric by $\nu$-balanced metrics is also less accurate. 
One should develop further techniques to approximate accurately the metric near singular points of the moduli space.

One can explain heuristically why, for a fixed value of $k$, this scheme is not as accurate near singular points. In particular, the limit $k\gg 1$ corresponds to the semiclassical limit, in K\"ahler geometric quantization, of $(X,\omega)$ with Planck's constant $\hbar=\frac{{\rm Vol}(X_t)^{1/n}}{k}$. Due to quantum uncertainty in regions of volume smaller than $\hbar^n$, one expects that accurate approximations of geometric features in $X$ occur when the size of such features is bigger than $\frac{{\rm Vol}(X_t)}{k^n}$. Therefore, as a singularity is a geometric object of zero volume, these numerical constructions should fail to give accurate approximations 
near singularities. This limitation to approximate non-smooth objects using polynomial approximations 
is similar to the so-called Runge phenomenon in numerical analysis.

\subsection{Extra remarks}

Another---more difficult and slower---way to approximate Weil-Petersson metrics involves evaluating the Weil-Petersson formula itself on a family of balanced metrics. 
In other words, instead of using Eq. \eqref{quantmet} to approximate the Weil-Petersson metric one could evaluate the WP metric itself on a family of metrics as
(see for instance \cite{ST})
\begin{equation}
\Upsilon_k(v_1,\, v_2) = \frac{1}{{\rm Vol}(X_{t})} \int_X 
v_1^a \bar{v}_2^{\bar{b}} g_t^{i\bar{\jmath}} \frac{\overline{\partial}}{\partial \overline{w}_{\bar{\jmath}}} \left(h_t^{-1} \frac{\partial h_t}{\partial t_a}  \right) \left ( \frac{\overline{\partial}}{\overline{\partial} \overline{w}_{\bar{\imath}}} \left(h_t^{-1} \frac{\partial h_t}{\partial t_{b}}  \right) \right)^\ast
\nu_{t}\wedge\overline{\nu}_{t},
\label{schumtoma}
\end{equation}
with $h_t$ the family of balanced metrics defined in Eq. \eqref{hermitianmetric}. 
As Eq. \eqref{schumtoma} becomes the definition of the Weil-Petersson metric when 
$h_t$ is the Calabi-Yau metric on the polarization,  one expects that if $h_t$ is balanced Eq. \eqref{schumtoma} will converge to the Weil-Petersson metric 
in the $k\to \infty$ limit. We have implemented an algorithm that computes this metric on the family of quintics that we have studied in this paper. 
We found that the numerical calculation itself is much slower than the numerical evaluation of Eq. \eqref{quantmet}. For instance, to compute Eq. \eqref{schumtoma} for $k=3$ took as much time as computing Eq. \eqref{quantmet} for $k=6$. Due to limitations with the speed of the numerical calculation, 
we decided not to resume a detailed analysis of a numerical method based on Eq \eqref{schumtoma}.
\begin{figure}
\begin{center}
    \includegraphics[width=3in]{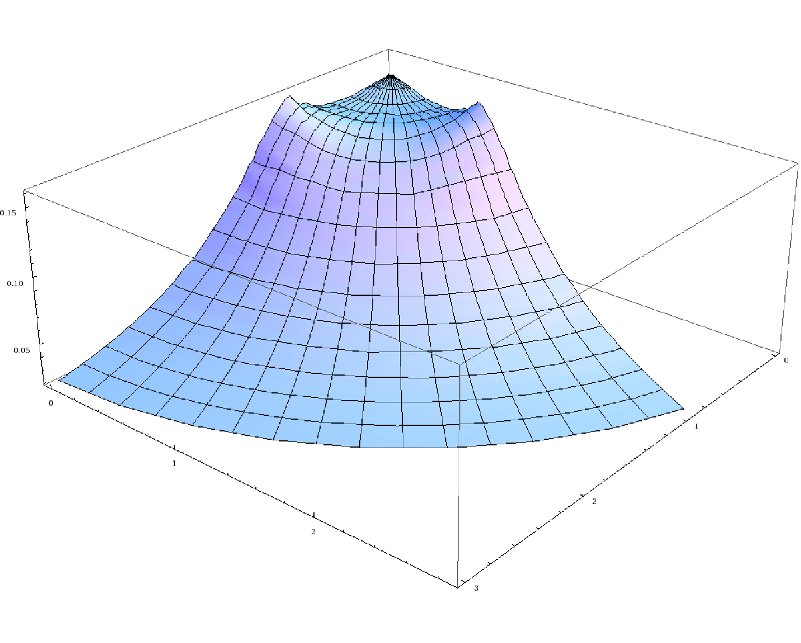}
\end{center}
  \caption{Approximation of the Weil-Petersson metric of Quintic 3-folds,  $Z_{0}^5+Z_{1}^5+Z_{2}^5+Z_{3}^5+Z_{4}^5 - 5t Z_{0} Z_{1} Z_{2} Z_{3} Z_{4}$, using Eq. \eqref{schumtoma} and a family of $k=1$ balanced metrics.}
    \label{harm}
\end{figure}
\medskip
%\par Many other problems could be studied in further depth using our approach to numerical Weil-Petersson metrics. For instance, as an application, one could try to estimate the Weil-Petersson volume of the moduli space of Calabi-Yau quintics.
%There are many open problems that one could study in further depth. For instance, it would be interesting to compute the quantized Weil-Petersson metrics
%on moduli spaces of stable vector bundles.

%In smooth points of the moduli space one should obtain very accurate approximations to the Weil-Petersson metric. Other interesting line of research consists in exploring the Weil-Petersson geometry itself on moduli spaces of polarized Calabi-Yau manifolds. For instance, one could estimate the Weil-Petersson volume of the moduli spaces in some special cases. 

\bigskip
\begin{center}
{\textsc{Acknowledgments}}
\end{center}
The authors are very grateful to S.K. Donaldson for many helpful conversations. They would also like to thank R.P.  Thomas for 
valuable suggestions. J. K. thanks  F. Bogomolov, R. Berman, and A. Soyeur for their suggestions. 
S.L. is particularly thankful to R. Karp for many detailed comments on this manuscript.


\begin{thebibliography}{DDDD}
%\bibitem[Bou]{Bou} T. Bouche, {\em Convergence de la m\'{e}trique de Fubini-Study d'un fibr\'{e} lin\'{e}aire positif}, Ann. Inst. Fourier 40, (1990).

\bibitem[BLY]{YauI}
J.-P. Bourguignon, P. Li and S.-T. Yau,
{\em Upper bound for the first eigenvalue of algebraic submanifolds}.
Comm. Math. Helv. 69, no. 1 (1994), 199-207. 


\bibitem[COGP]{Can}
P. Candelas, X. C. De La Ossa, P. S. Green, L. Parkes,
{\em A Pair of Calabi-Yau manifolds as an exactly soluble superconformal theory}.
Nucl. Phys. B359 (1991), 21-74. 

%\bibitem[DLM]{DLM} X. Dai, K. Liu, and X. Ma, {\em On the asymptotic expansion of Bergman kernel}, J. Differential Geom. 72 (2006), no. 1, 1–41.

%\bibitem[Dem1]{Dem1} J.P. Demailly, {\em Champs magn\'etiques et in\'egalit\'es de Morse pour la d''-cohomologie}, Ann. Inst. Fourier, Grenoble, 35-4 (1985), 189-229
%\bibitem[Dem2]{Dem2} J.P. Demailly, {\em Holomorphic Morse inequalities}, Lectures given at the AMS Summer Institute on Complex Analysis held in Santa Cruz, July 1989, Proceedings of Symposia in Pure Mathematics, Vol. 52, Part 2 (1991), 93-114

%\bibitem[Don1]{DonIV}
%S.K. Donaldson,
%{\em Infinite determinants, stable bundles and curvature}. 
%Duke Math. Jour. {\bf 54} (1987), 231-247.  


\bibitem[Don1]{DonI}
S.K. Donaldson,
{\em Remarks on gauge theory, complex geometry and 4-manifold topology,} in
{\em Fields Medallists Lectures}. World Sci. Ser. 20th Century Math. {\bf 5},
World Sci. Publishing, River Edge, NJ, 384-403, (1997).


%\bibitem[Don3]{DonIII}
%S.K. Donaldson,
%{\em Geometry in Oxford c. 1980--85. Sir Michael Atiyah: a great mathematician of the twentieth century}. 
%Asian J. Math. 3 (1999), no. 1, xliii--xlvii. 


\bibitem[Don2]{DonII}
S.K. Donaldson,
{\em Scalar curvature and projective embeddings, I}.
Jour. Diff. Geom. {\bf 59}, 479-522, (2001). 

\bibitem[Don3]{DonV} S.K.\ Donaldson, {\em Scalar curvature and projective embeddings II}, 
Quaterly Jour. Math 56 (2005).


\bibitem[Don4]{Don}
S.K. Donaldson,
{\em Some numerical results in complex differential geometry}, Pure Appl Math Q, Vol:5, 571--618 (2009)



\bibitem[DKLR]{Dou}
M.R. Douglas, R.L. Karp, S. Lukic, R. Reinbacher, 
{\em Numerical Calabi-Yau metrics}.
J. Math. Phys. 49, (2008). 

\bibitem[DL]{DL}
M.R. Douglas and Z. Lu,
{\em On the Geometry of Moduli Space of Polarized Calabi-Yau manifolds}.
RIMS Kokyuroku 1487, (2006).

\bibitem[Fuj]{Fuj} A. Fujiki, {\em Moduli space of polarized algebraic manifolds and K\"ahler metrics}, 
Sugaku Exposition Vol 5, 2 (1992).

\bibitem[FS]{F-S} A. Fujiki, G. Schumacher, {\em The Moduli space of Extremal Compact Kahler Manifolds
and Generalized Well-Petersson Metrics}, Publ. RIMS, Kyoto Univ. 26 (1990), 101-183

%\bibitem[Gie]{Gie} D. Gieseker, {\em On the moduli of vector bundles on an algebraic surface,} Ann. Math. 106(1977), 45-60.

\bibitem[GSW]{GSW}
M.B. Green, J.H. Schwarz and E. Witten, 
{\em Superstring Theory: Volume 2}.
Cambridge University Press, (1988). 

%\bibitem[Hit]{Hi} N. Hitchin, {\em Harmonic spinors}, Adv. in Math. 14, 1–55 (1974).

%\bibitem[HL]{HL} D. Huybrechts and M. Lehn, {\em The Geometry of Moduli Spaces of Sheaves,} Vieweg 1997. 
\bibitem[Kel]{Ke} J. Keller {\em Ricci iterations on Kähler classes.}, J. Inst. Math. Jussieu 8, no. 4, 743--768 (2009)

%\bibitem[LT]{LT} M. Lubke, A. Telamn, {\em The Kobayashi-Hitchin Correspondence}, World Scientific Pub Co Inc (1995).

\bibitem[MM]{M-M}  X. Ma and G. Marinescu, {\em Holomorphic Morse inequalities and Bergman kernels}, Progress in Mathematics, 254, Birkhauser Verlag, (2007)

\bibitem[Mor]{Mor} D.R. Morrison, {\em Picard-Fuchs equations and mirror maps for hypersurfaces}, AMS/IP, Studies in Adv. Math. 
Mirror symmetry I, (1993).

\bibitem[San]{Sa} Y. Sano, {\em Numerical algorithm for finding balanced metrics},  Osaka J. of Maths. 43, 3 (2006).

\bibitem[ST]{ST}
G. Schumacher and M. Toma, {\em On the Petersson-Weil metric for the moduli space of Hermite-Einstein bundles and its curvature},
Math. Ann. 293, 101-107, (1992).

\bibitem[SZ]{SZ}
B. Shiffman and S. Zelditch, 
{\em Distribution of zeros of random and quantum
chaotic sections of positive line bundles,} Comm. Math. Phys. {\bf
200} (1999), no. 3, 661-683.


%\bibitem[Tho]{Tho}
%R.P. Thomas,
%{\em Notes on GIT and symplectic reduction for bundles and varieties}.
%Arxiv preprint {\tt math.AG/0512411}, (2005).


\bibitem[Tia]{Tian}
G. Tian,
{\em Smoothness of the universal deformation space of compact Calabi-Yau manifolds and its Weil-Petersson metric}.
Math. Aspects of String Theory. Yau S.-T. (ed.) Singapore: World Scientific (1988).


\bibitem[Tod]{Todorov}
A.N. Todorov,
{\em The Weil-Petersson geometry of the moduli space of $SU(n\geq 3)$ Calabi-Yau manifolds I}.
Commun. Math. Phys. {\bf 126}, 325-346 (1989).


%\bibitem[UY]{UY}
%K. Uhlenbeck, S.-T. Yau,
%{\em On the existence of hermitian-yang-mills connections in stable vector bundles}.
%Comm. Pure Appl. Math. 39 Issue S1 (1986), S257 - S293

\bibitem[Vie]{Vie} E. Viehweg, {\em  Quasi-projective moduli for polarized manifolds}, volume 30 of
Ergebnisse der Mathematik und ihrer Grenzgebiete (3) [Results in Mathematics and Related Areas (3)]. Springer-Verlag, Berlin, 1995.


%\bibitem[Wan1]{WangI} X. Wang, {\em Balance point and stability of vector bundles over a projective manifold}. Math. Res. Lett. 9 (2002), no. 2-3, 393--411. 

%\bibitem[Wan2]{WangII} X. Wang,  {\em Canonical metrics on stable vector bundles}, Comm. Anal. Geom. 13, (2005).

\bibitem[Yau]{Yau}
S.-T. Yau,
{\em On the Ricci curvature of a compact Kahler manifold and the complex Monge-Ampere equation I}.
Comm. pure appl. math, (1978).

\bibitem[Donag]{Donag}
R. Donagi, A. Lukas, B. A. Ovrut and D. Waldram
{\em Holomorphic vector bundles and non-perturbative vacua in M-theory}.
Journal of High Energy Physics (1999), 06, 034.


\end{thebibliography}
\end{document}